\documentclass[a4paper, 12pt]{amsart}
\usepackage[utf8]{inputenc}
\usepackage{amsopn,amssymb}
\usepackage[dvipsnames]{xcolor}
\usepackage{fullpage,amsfonts,mathtools}
\usepackage{enumitem, comment, lineno}
\usepackage[colorlinks=true, 
pdfstartview=FitV, linkcolor=blue, 
citecolor=blue, urlcolor=blue, backref = page]{hyperref}

\usepackage[msc-links, alphabetic]{amsrefs}

\usepackage[capitalise]{cleveref}
\usepackage{cancel}

\newtheorem{theorem}{Theorem}[section]
\newtheorem{lemma}[theorem]{Lemma}
\newtheorem{proposition}[theorem]{Proposition}
\newtheorem{corollary}[theorem]{Corollary}

\theoremstyle{definition}
\newtheorem{definition}{Definition}[section]

\theoremstyle{remark}

\crefname{inequality}{Inequality}{Inequalities}
\allowdisplaybreaks

\DeclareMathOperator*{\E}{\mathbb{E}}

\title{Furstenberg--S\'{a}rk\"{o}zy theorem over number fields}
\author{Dev Ranjan Pandey}
\address{Department of Mathematics, Indian Institute of Science 
Education and Research Bhopal, Bhopal Bypass Road, Bhauri, Bhopal 462066, Madhya Pradesh, India}
\email{dev23@iiserb.ac.in}

\author{Jyoti Prakash Saha}
\address{Department of Mathematics, Indian Institute of Science 
Education and Research Bhopal, Bhopal Bypass Road, Bhauri, Bhopal 462066, Madhya Pradesh, India}
\email{jpsaha@iiserb.ac.in}


\begin{document}

\begin{abstract}
We introduce the notion of intersective polynomials having coefficients in 
the ring of integers $\mathcal{O}_K$ of a number field $K$, 
and define a notion of upper density of subsets of $\mathcal{O}_K$. 
We prove that given any intersective polynomial $p(x)$ over $\mathcal{O}_K$, 
every subset $A$ of $\mathcal{O}_K$ of positive upper density contains 
two distinct elements whose difference is equal to $p(x)$ for some 
element $x$ in $\mathcal{O}_K$.  
Moreover, we obtain a quantitative version of this result.  
The proof is motivated by an argument due to Lucier, 
and the Fourier-free proof of the Furstenberg--S\'{a}rk\"{o}zy theorem 
over the integers by Green, Tao and Ziegler.
\end{abstract}

\maketitle

\section{Introduction}
\label{section:introduction}

\subsection{Motivation}

A conjecture by Lov\'asz states that every subset $A$ of the integers, with 
positive upper density, contains two distinct elements in $A$ such that 
they differ by a perfect square.
In 1977, Furstenberg \cite{Furstenberg} proved this conjecture by 
using ergodic theoretic methods.
Soon after, S\'{a}rk\"{o}zy independently proved a quantitative
and a stronger version of this result using Fourier analytic methods.
\begin{theorem} [S\'ark\"ozy \cite{Sarkozy}]
Let $A$ be a subset of the integers lying in the interval $[1,N]$ such that
the set $A - A$ does not contain any nonzero perfect square.
Then 
\[
  \frac{|A|}{N}
  \ll
  \left(
    \frac{(\log \log N)^2}{\log N} 
  \right)^{1/3}.
\]
\end{theorem}
In the subsequent years, improved upper bounds have been obtained 
for the density of the subsets $A$ of $\left\{ 1, 2, \dots, N \right\}$
such that the difference set $A - A$ avoids the nonzero perfect squares, 
by Pintz, Steiger, and Szemer\'edi \cite{PSS},
Lucier \cite{Lucier},
Lyall and Magyar \cite{LyallMagyar},
Hamel, Lyall, and Rice \cite{HLR},
Rice \cite{Rice},
Bloom and Maynard \cite{BloomMaynard},
among others. 
Recently, Green and Sawhney \cite{GreenSawhney} have announced the 
best known upper bound, stating that 
for any subset $A$ of $\{1,2,\dots,N\}$, 
containing no two elements whose difference is a nonzero perfect square,
\[
  \frac{|A|}{N}
  \ll
  e^{-c \sqrt{\log N}}
\]
holds, where $c$ is a positive constant.
Green, Tao, and Ziegler \cite{TaoBlog} obtained a Fourier-free proof of 
the Furstenberg--S\'ark\"ozy theorem, using the 
Cauchy--Schwarz inequality. 

In 1994, Balog, Pelik\'an, Pintz, and Szemer\'edi 
\cite{BPPS} proved that for any integer $\kappa \geq 1$,
every subset of the integers with positive upper density
contains two distinct elements such that their difference is the $\kappa$-th power 
of an integer.
More precisely, they proved that if $A$ is a subset of $\{1,2,\dots,N\}$
such that no two of its elements differ by a nonzero perfect $\kappa$-th power, then 
\[
  \frac{|A|}{N}
  \ll_\kappa
  \left(
    \log N
  \right)^{-\frac{1}{4}\log \log \log \log N}.
\]
A work of Kamae and Mend\`es~France \cite{KamaeMendes} shows that
if $p(x)$ is an intersective polynomial over the integers, 
then every subset of the integers with positive upper density contains two 
distinct elements differing by $p(n)$ for some integer $n$.
Lucier proved 
the following quantitative result, 
which implies the above result of Kamae and Mend\`es~France.

\begin{theorem} [Lucier \cite{Lucier}]
  Let $p(x) \in \mathbb{Z}[x]$ be an intersective polynomial 
  of degree $\kappa \geq 2$, 
  and let $A$ be a subset of $\{1,2,\dots,N\}$ such that it does not contain
  two distinct elements with difference equal to $p(n)$ for some integer $n$.
  Then 
  \[
    \frac{|A|}{N}
    \ll_p
    \frac{(\log \log N)^{\mu/(\kappa-1)}}{(\log N)^{1/(\kappa-1)}}
  \]
  holds,
  where $\mu = 3$ when $\kappa = 2$, and $\mu = 2$ when $\kappa \geq 3$.
\end{theorem}
In 2009, Lyall and Magyar \cite{LyallMagyar}, and in 2013,
Hamel, Lyall, and Rice \cite{HLR} 
have obtained better bounds, for the polynomials with constant term zero,
and for the intersective polynomials of degree two, respectively.
In 2019, Rice \cite{Rice} has further improved the bound for the 
intersective polynomials of arbitrary degree.
Bloom and Maynard \cite{BloomMaynard} 
proved that 
any subset $A$ of $\{1, 2, \dots, N\}$ without square differences
satisfy 
\[
  \frac{|A|}{N}
  \ll
  \left(
    \log N
  \right)^{-c \log \log \log N},
\]
where $c$ is a positive absolute constant.
Arala \cite{Arala} has extended this result to 
intersective polynomials of arbitrary degree.

\subsection{Result obtained}

Let $K$ be a number field of degree $d$ over $\mathbb{Q}$, and
$\mathcal{O}_K$ be the ring of integers of $K$.
Fix a $\mathbb{Z}$-basis $\{e_1, e_2, \dots, e_d\}$ of $\mathcal{O}_K$.
For any nonnegative real number $N$, let $S(N)$ denote the set 
\[
  \left\{ \sum_{i=1}^{d} a_i e_i : a_i \in \mathbb{Z} \cap [-N,N] \right\},
\]
and for any subset $A$ of $\mathcal{O}_K$, define the upper density of $A$
in $\mathcal{O}_K$ as
\[
  \overline{\delta}(A) \coloneq
  \limsup_{N \to \infty} 
  \frac{|A \cap S(N)|}{|S(N)|}.
\]

Note that if $p(x) \in \mathbb{Z}[x]$ is an intersective polynomial over 
the integers, then by the pigeonhole principle, one can show that 
the Furstenberg--S\'ark\"ozy theorem over the integers implies that
for every integer $1 \leq j \leq d$,
each subset $A$ of $\mathcal{O}_K$ of positive upper density contains two 
distinct elements such that their difference is $e_j p(n)$ for 
some integer $n$.

Observe that every nonzero ideal of $\mathcal{O}_K$ is of positive
upper density.
If $p(x) \in \mathcal{O}_K[x]$ is a polynomial for which 
there exists a nonzero ideal $\mathfrak{a}$ of $\mathcal{O}_K$ 
which does not contain any nonzero value of $p(x)$, then
the difference set $\mathfrak{a} - \mathfrak{a}$ does not contain
any nonzero value of the polynomial $p(x)$,
although the upper density of $\mathfrak{a}$ in $\mathcal{O}_K$ is positive.
This leads to the definition of the intersective polynomials over 
the ring $\mathcal{O}_K$ (see \cref{definition of intersective polynomial}).
We obtain \cref{Thm:main result}, which establishes 
the Furstenberg--S\'ark\"ozy theorem for the intersective polynomials 
over the ring $\mathcal{O}_K$.

\begin{theorem}
\label[theorem]{Thm:main result}
  Let $q(x) \in \mathcal{O}_K[x]$ be an intersective polynomial of degree 
  $\kappa \geq 2$ and $N$ be a positive integer. 
  Let $A$ be a subset of $S(N)$ which does not contain two 
  distinct elements whose difference is $q(x)$ 
  for some element $x$ of $\mathcal{O}_K$.
  Then the density of $A$ in $S(N)$ satisfies
  \[
    \frac{|A|}{|S(N)|}
    \ll_{q, d, E}
    \frac{1}{(\log \log N)^\varepsilon}
  \]
  where 
  \(\varepsilon = \frac{1}{2^{\kappa - 1} + 1}\).
\end{theorem}

\subsection{Outline of the proof}

Our proof is motivated by 
an argument of Lucier \cite{Lucier} 
used in 
proving the 
Furstenberg--S\'{a}rk\"{o}zy theorem 
for intersective polynomials, 
and the Fourier-free proof of this theorem 
by Green, Tao, and Ziegler \cite{TaoBlog} for the perfect squares.
We provide a brief outline of our proof below.

Let $q(x) \in \mathcal{O}_K[x]$ be an intersective polynomial 
of degree $\kappa \geq 2$ over $\mathcal{O}_K$, 
and let $A$ be a subset of $S(N)$ of density $\delta$.
Suppose the difference set $A - A$ contains no nonzero value of $q(x)$.
Assume that $N \geq N_{\delta, q}$, where $N_{\delta, q}$ 
is a suitable positive integer, and it is large in terms of $\delta$, 
the degree of $K$, and the polynomial $q(x)$.

In \cref{section:density increment}, 
we use the strategy of \cite{TaoBlog}.
Let \(f\) denote the balanced function 
of the subset \(A\) of \(S(N)\).
First we show that \cref{avoidance in difference set implies correlation}
holds, which proves that 
the correlation between 
$f$ and the composition of $f$ 
and a shift map by $q(x)$
is large in terms of $\kappa$ and $\delta$.
In \cref{degree reduction of polynomial}, we prove a degree lowering result,
which can be applied iteratively 
to obtain \cref{Linearization},
which states that 
$f$ has a large correlation 
with the composition of $f$ 
and the shift map by a linear polynomial,
in terms of $\kappa$ and $\delta$. 
Moreover, the leading coefficient of this linear polynomial, 
denoted by $\alpha$, 
is not ``too large''.
Applying the pigeonhole principle, 
we prove \cref{density increment lemma}, 
which shows that for a suitable element
$n$ in $\mathcal{O}_K$, 
the density of $A$ in the set 
$n + \alpha S(N^{1/2\kappa})$ 
is at least $\delta + c \delta^C$, 
for some positive constants $c, C$, 
depending only on $\kappa$.

In \cref{section:intersective polynomials}, 
we define the notion of intersective polynomials over $\mathcal{O}_K$
(see \cref{definition of intersective polynomial}).
A polynomial with coefficients in $\mathcal{O}_K$ is said to be 
\emph{intersective} if it admits a root modulo 
every nonzero element of $\mathcal{O}_K$.
It follows that every intersective polynomial has a root
in the $\mathfrak{p}$-adic completion of $\mathcal{O}_K$
for every nonzero prime ideal $\mathfrak{p}$ of $\mathcal{O}_K$.
For every nonzero prime ideal $\mathfrak{p}$, 
we choose a root $z_{\mathfrak{p}}$ of $q(x)$ in the $\mathfrak{p}$-adic 
completion.
Using the Chinese remainder theorem, we show in 
\cref{existence of r_a} that for any nonzero element
$\alpha$ in $\mathcal{O}_K$, 
if $\mathfrak{a}$ denotes the ideal generated by $\alpha$,
then the polynomial $q(x)$ admits a root 
$r_\alpha$ modulo $\mathfrak{a}$, which can be bounded in terms of $\alpha$.
Moreover, $r_\alpha$ is congruent to $z_{\mathfrak{p}}$ 
modulo $\mathfrak{p}^n$ for every nonzero prime ideal $\mathfrak{p}$
and positive integer $n$ such that $\mathfrak{p}^n$ divides $\mathfrak{a}$.
Following Lucier, 
for every nonzero element $\alpha$ of $\mathcal{O}_K$, 
we define an auxiliary polynomial $q_{\alpha}(x)$ of degree $\kappa$,
whose coefficients admit an upper bound 
in terms of those of $q(x)$ and $\alpha$.
Moreover, the polynomial $q_{\alpha}(x)$ is intersective
as proved in 
\cref{intersectivity of auxiliary polynomial}.

Combining
\cref{density increment lemma} and
\cref{intersectivity of auxiliary polynomial}, 
we prove in 
\cref{Combination of DIL and Lucier}
that 
some subset $A'$ of $S(N^{1/2\kappa})$ 
has density at least $\delta + c \delta^C$, 
and the difference set $A' - A'$ avoids the nonzero values of 
the intersective polynomial $q_{\alpha}(x)$.
We also show that the coefficients of $q_{\alpha}(x)$ 
admit a suitable upper bound, which is not too large in terms of $N$
and the coefficients of $q(x)$.
Next, we provide a proof of 
\cref{Thm:main result}. 
We show that under the assumption $N \geq N_{\delta, q}$, 
the above argument can be iterated enough times 
to arrive at a subset of a set with density larger than $1$.
This implies that $N < N_{\delta, q}$ holds, which yields the 
bound stated in \cref{Thm:main result}.

\section{Preliminaries}
\label{section:preliminaries}

Throughout this article, $K$ denotes
a number field having degree $d$ over $\mathbb{Q}$.
Its ring of integers is denoted by $\mathcal{O}_K$.
Let $\{e_1, e_2, \dots, e_d\}$ denote a \(\mathbb{Z}\)-basis of \(\mathcal{O}_K\).

For a positive real number $N$, we define the set
\[
  S(N) \coloneq 
  \left\{ \sum_{i=1}^{d} a_i e_i : a_i \in \mathbb{Z} \cap [-N,N] \right\}.
\]
Note that $|S(N)| = (2 \left\lfloor N \right\rfloor + 1)^d$, 
and hence, 
\((2N-1)^d \leq |S(N)| \leq (2N+1)^d\) 
holds. 
The density of a subset $A$ of $\mathcal{O}_K$ in $S(N)$ is defined as
\[
  \delta_N (A) \coloneq \frac{|A \cap S(N)|}{|S(N)|},
\]
and the upper density of $A$ is defined as
\[
  \overline{\delta}(A) \coloneq \limsup_{N \to \infty} \delta_N (A).
\]
For a subset $A$ of $S(N)$ of density $\delta$,
we denote the \emph{balanced function} of $A$ by 
$f : \mathcal{O}_K \to \mathbb{R}$, which is defined as
\[
  f(n) \coloneq \mathbf{1}_{A}(n) - \delta \mathbf{1}_{S(N)}(n).
\]

For any element of \(\alpha\) of \(\mathcal{O}_K\), 
let \(M_\alpha\) denote the smallest nonnegative 
integer such that \(\alpha\) lies in \(S(M_\alpha)\).
In other words, if \(\alpha = \sum_{i = 1}^d a_i e_i\), 
where $a_i \in \mathbb{Z}$ for all $1 \leq i \leq d$,
then 
\[
  M_\alpha = \max_{1 \leq i \leq d} |a_i|.
\]
Let \(E\) denote the integer defined by 
\[
  E \coloneq 
  \max_{1 \leq i \leq j \leq d}
  M_{e_{i} e_{j}}
\]
Note that if $x \in S(N_1)$ and $y \in S(N_2)$, then the product $xy$ lies in 
the set $S(d^2 E N_1 N_2)$. 
In other words, 
\begin{equation}
  \label{Eqn: product of finite lattices}
  S(N_1) \cdot S(N_2)
  \subseteq S(d^2 E N_1 N_2)
\end{equation}
holds. 
For any polynomial $p(x) = c_0 + c_1 x + \cdots + c_\kappa x^\kappa$ 
of degree $\kappa$ in $\mathcal{O}_K[x]$,
let \(M_p\) denote the integer defined by 
\[
  M_p 
  \coloneq
  \max_{0 \leq i \leq \kappa} M_{c_i}.
\]

For any nonnegative integer $j$, some subsets $X, H$ of $\mathcal{O}_K$, 
and elements $x \in X$ and $\mathbf{h} = (h_1, \dots, h_j) \in H^j$,
we write
\[
  \sigma_j(x, \mathbf{h}) \coloneq x + h_1 + \dots + h_j.
\] 
Note that if $k$ lies in $H$, 
then for any element $\mathbf{h} = (h_1, \dots, h_j) \in H^j$, 
the element $\mathbf{h'} \coloneq (h_1, \dots, h_j, k)$ lies in $H^{j+1}$, and 
\[
  \sigma_{j+1}(x, \mathbf{h'}) = \sigma_{j}(x, \mathbf{h}) + k
\]
holds for every $x\in X$.

Let us set the constants $C_1 = (\kappa + 1) (\kappa d^2 E)^{\kappa}$, 
$C_2 = 2(2 d^2 E)^\kappa$, 
$C_3 = \kappa 2^\kappa (d^2 E)^{3\kappa d}$
and $C_4 = C_2^{4\kappa^2 d} C_3$
depending on $\kappa$, $d$ and $E$ only.

The following proposition gives some useful properties of the balanced function $f$
of a subset of $S(N)$.

\begin{proposition} \label[proposition]{properties of f}
  Let $N$ be a positive integer and $A$ be a subset of $S(N)$ of density 
  $\delta$.
  Then the balanced function $f$ of $A$ has the following properties.
  \begin{enumerate}
    \item The function $f$ is $1$-\emph{bounded}.
    \item $\E_{n \in S(N)}f(n) = 0$.
    \item $\E_{n \in S(N)}\left\vert f(n) \right\vert^2 \leq \delta$.
  \end{enumerate}
\end{proposition} 

\begin{proof}
  Given any element $n$ in $\mathcal{O}_K$, note that
  $f(n)$ lies in the set  $\left\{0,1 - \delta, -\delta\right\}$, 
  implying that $f$ is $1$-bounded.
  Also, observe that
  \begin{align*}
    \E_{n \in S(N)}f(n) 
    =
    \frac{1}{|S(N)|} 
    \sum_{n \in S(N)} 
    \left(\mathbf{1}_{A}(n) - \delta \mathbf{1}_{S(N)}(n)\right)
    =
    \frac{1}{|S(N)|}
    \left(\vert A \vert - \delta \cdot |S(N)|\right)
    =
    0
  \end{align*}
  holds.
  This proves the second statement.
  Note that 
  \begin{align*}
    \E_{n \in S(N)}\left\vert f(n) \right\vert^2
    &
    =
    \E_{n \in S(N)}
    \left(\mathbf{1}_{A}(n) - \delta1_{S(N)}(n)\right)^2\\
    &
    =
    \E_{n \in S(N)}
    \left(
      \left(\mathbf{1}_{A}(n)\right)^2 +
      \delta^2 \left(1_{S(N)}(n)\right)^2 -
      2\delta\mathbf{1}_{A}(n) \cdot 1_{S(N)}(n)
    \right)\\
    &
    =
    \frac{1}{|S(N)|}
    \left(
      \left\vert A \right\vert +
      \delta^2 \left\vert S(N) \right\vert -
      2\delta \left\vert A \right\vert
    \right)\\
    &
    =
    \delta + \delta^2 - 2\delta^2\\
    &
    \leq
    \delta
  \end{align*}
  holds, which
  proves the final statement.
\end{proof}

\section{Density increment}
\label{section:density increment}

In this section, our goal is to prove a density increment result, stated in 
\cref{density increment lemma}. It implies that,
if for a sufficiently large integer $N$, a subset $A$ of $S(N)$ has no two distinct 
elements which differ by a value of a fixed  
nonconstant polynomial in $\mathcal{O}_K [x]$ of degree $\kappa$,
then there exists a translation of a dilation of the set
$S(N^{1/2\kappa})$, which is contained in $S(N)$, and
on which the density of the set $A$ is increased suitably.
In the following lemma, we prove that the correlation between 
the balanced function $f$ of a subset $A$ of $S(N)$ and the composition
of $f$ with a shift by a polynomial $p(x)$ is large, provided 
the difference set $A-A$ does not contain any nonzero value of $p(x)$.

\begin{lemma} \label[lemma]{avoidance in difference set implies correlation}
  Let $p(x)$ be a nonconstant polynomial 
  of degree $\kappa$ in $\mathcal{O}_K[x]$, and $\delta_0 \in (0,1]$.  
  Let $N$ be a positive integer and $A$ be a subset of $S(N)$ of density
  $\delta \geq \delta_0$.
  Suppose the difference set $A - A$ does not contain the nonzero values of $p(x)$.
  If 
  \[
    N
    \geq
    \left(
      \frac{2^{d + 5} C_1 M_p}
      {\delta_0^{2}}
    \right)^{2}
  \]
  then given any positive integer $r$ the balanced function $f$ of $A$ satisfies
  \[
    \left\vert 
      \E_{n \in S(N)}
      \E_{x \in S(N^{1/{2\kappa}})}
      \E_{\mathbf{h} \in S(N^{1/4\kappa^{r}})^{\kappa-1}}
      f(n)
      f\left(
        n + p\left(\sigma_{\kappa-1}(x, \mathbf{h})\right)
      \right)
    \right\vert  
    \geq
    \frac{\delta^{2}}{2}.
  \]
\end{lemma}

\begin{proof}
  Let $r$ be a positive integer and
  $N
    \geq
    \left(
      \frac{2^{d + 5} C_1 M_p}
      {\delta_0^{2}}
    \right)^{2}
  $ holds.
  The polynomial $p(x)$ can have at most $\kappa$ many roots in
  $\mathcal{O}_K$,
  and for any element $n$ in $\mathcal{O}_K$,
  there exist  at most $|S(N^{1/4\kappa^{r}})|^{\kappa-1}$ many elements 
  in the set $S(N^{1/2\kappa}) \times S(N^{1/4\kappa^{r}})^{\kappa-1}$
  such that the sum of its components is equal to $n$.
  Therefore, we get that 
  \[
    \E_{n \in S(N)}
    \E_{x \in S(N^{1/2\kappa})}
    \E_{\mathbf{h} \in S(N^{1/4\kappa^{r}})^{\kappa-1}}
    \mathbf{1}_{A}(n)
    \mathbf{1}_{A}\left(n + p \left(\sigma_{\kappa-1}(x, \mathbf{h})\right)\right)
    \leq
    \frac{\kappa}{|S(N^{1/2\kappa})|}.
  \]
  Note that 
  $N \geq \left(\frac{4\kappa}{\delta^2}\right)^{2\kappa}$
  holds, which yields that
  \begin{equation} \label{first essential}
    \E_{n \in S(N)}
    \E_{x \in S(N^{1/2\kappa})}
    \E_{\mathbf{h} \in S(N^{1/4\kappa^{r}})^{\kappa-1}}
    \mathbf{1}_{A}(n)
    \mathbf{1}_{A}\left(n + p \left(\sigma_{\kappa-1}(x, \mathbf{h})\right)\right)
    \leq
    \frac{\delta^2}{4}.
  \end{equation}
  Consider the balanced function $f$ of the subset $A$ of $S(N)$ and observe that
  \begin{equation}
  \label{no structure implies balanced function property}
  \begin{aligned}
    &
    \E_{n \in S(N)}
    \E_{x \in S(N^{1/2\kappa})}
    \E_{\mathbf{h} \in S(N^{1/4\kappa^{r}})^{\kappa-1}}
    f(n)
    f\left(n + p \left(\sigma_{\kappa-1}(x, \mathbf{h})\right)\right)\\    
    =
    &
    \E_{n \in S(N)}
    \E_{x \in S(N^{1/2\kappa})}
    \E_{\mathbf{h} \in S(N^{1/4\kappa^{r}})^{\kappa-1}}
    \mathbf{1}_{A}(n)
    \mathbf{1}_{A}
    \left(n + p \left(\sigma_{\kappa-1}(x, \mathbf{h})\right)\right)\\
    &
    -
    \E_{n \in S(N)}
    \E_{x \in S(N^{1/2\kappa})}
    \E_{\mathbf{h} \in S(N^{1/4\kappa^{r}})^{\kappa-1}}
    \mathbf{1}_{A}(n)
    \delta \mathbf{1}_{S(N)}
    \left(n + p \left(\sigma_{\kappa-1}(x, \mathbf{h})\right)\right)\\
    &
    -
    \E_{n \in S(N)}
    \E_{x \in S(N^{1/2\kappa})}
    \E_{\mathbf{h} \in S(N^{1/4\kappa^{r}})^{\kappa-1}}
    \delta \mathbf{1}_{S(N)}(n)
    \mathbf{1}_{A}
    \left(n + p \left(\sigma_{\kappa-1}(x, \mathbf{h})\right)\right)\\
    &
    +
    \E_{n \in S(N)}
    \E_{x \in S(N^{1/2\kappa})}
    \E_{\mathbf{h} \in S(N^{1/4\kappa^{r}})^{\kappa-1}}
    \delta \mathbf{1}_{S(N)}(n)
    \delta \mathbf{1}_{S(N)}
    \left(n + p \left(\sigma_{\kappa-1}(x, \mathbf{h})\right)\right).
  \end{aligned}
  \end{equation}
  Note that
  \begin{equation} \label{second essential}
    \E_{n \in S(N)}
    \E_{x \in S(N^{1/2\kappa})}
    \E_{\mathbf{h} \in S(N^{1/4\kappa^{r}})^{\kappa-1}}
    \delta \mathbf{1}_{S(N)}(n)
    \delta \mathbf{1}_{S(N)}
    \left(n + p \left(\sigma_{\kappa-1}(x, \mathbf{h})\right)\right)
    \leq
    \delta^2
  \end{equation}
  holds. 
  For every nonnegative integer $j \leq \kappa-1$ and 
  $(x,\mathbf{h})$ in $S(N^{1/2\kappa}) \times S(N^{1/4\kappa^{r}})^{j}$, 
  the polynomial $p(x)$ satisfies
  \[
    p\left(\sigma_{j}(x, \mathbf{h})\right)
    \in 
    S(C_1 M_p N^{1/2}),
  \]
  and consequently, the set $S\left( N - C_1 M_p N^{1/2} \right)$
  is contained in the sets $S(N) + p\left(\sigma_{j}(x, \mathbf{h})\right)$
  and $S(N) - p\left(\sigma_{j}(x, \mathbf{h})\right)$.
  Therefore,
  \begin{align*}
    &
    \E_{n \in S(N)}
    \E_{x \in S(N^{1/2\kappa})}
    \E_{\mathbf{h} \in S(N^{1/4\kappa^{r}})^{\kappa-1}}
    \mathbf{1}_{A}(n)
    \delta \mathbf{1}_{S(N)}
    \left(n + p \left(\sigma_{\kappa-1}(x, \mathbf{h})\right)\right)\\
    \geq
    &
    \frac{\delta}{|S(N)|}
    \E_{x \in S(N^{1/2\kappa})}
    \E_{\mathbf{h} \in S(N^{1/4\kappa^{r}})^{\kappa-1}}
    \sum_{n \in S(N)}
    \mathbf{1}_{A \cap S(N - C_1 M_p N^{1/2})}\left(n\right)\\
    \geq
    &
    \frac{\delta}{|S(N)|}
    \E_{x \in S(N^{1/2\kappa})}
    \E_{\mathbf{h} \in S(N^{1/4\kappa^{r}})^{\kappa-1}}
    \left(
      \left|A\right| - \left|S(N) \setminus
      S\left(N - C_1 M_p N^{1/2} \right)\right|
    \right)\\
    \geq
    &
    \delta^2 -
    \left(
      \frac{|S(N)| - \left|S\left( N - C_1 M_p N^{1/2} \right)\right|}{|S(N)|}
    \right)
  \end{align*}
  holds. 
  Note that 
  $
    N \geq
    \left(
      \frac{2^{d+5} C_1 M_p}{\delta^2}
    \right)^2
  $
  implies that
  \begin{equation} \label{third essential} 
    \E_{n \in S(N)}
    \E_{x \in S(N^{1/2\kappa})}
    \E_{\mathbf{h} \in S(N^{1/4\kappa^{r}})^{\kappa-1}}
    \mathbf{1}_{A}(n)
    \delta \mathbf{1}_{S(N)}
    \left(n + p \left(\sigma_{\kappa-1}(x, \mathbf{h})\right)\right)
    \geq
    \frac{7}{8}\delta^2
  \end{equation}
  holds. Similarly, we deduce that
  \begin{equation} \label{fourth essential}
  \begin{aligned}
    &
    \E_{n \in S(N)}
    \E_{x \in S(N^{1/2\kappa})}
    \E_{\mathbf{h} \in S(N^{1/4\kappa^{r}})^{\kappa-1}}
    \delta \mathbf{1}_{S(N)}(n)
    \mathbf{1}_{A}\left(n + p \left(\sigma_{\kappa-1}(x, \mathbf{h})\right)\right)\\
    =
    &
    \frac{\delta}{|S(N)|}
    \E_{x \in S(N^{1/2\kappa})}
    \E_{\mathbf{h} \in S(N^{1/4\kappa^{r}})^{\kappa-1}}
    \sum_{n \in S(N) + p \left(\sigma_{\kappa-1}(x, \mathbf{h})\right)}
    \mathbf{1}_{A}\left(n\right)\\
    \geq
    &
    \frac{\delta}{|S(N)|}
    \E_{x \in S(N^{1/2\kappa})}
    \E_{\mathbf{h} \in S(N^{1/4\kappa^{r}})^{\kappa-1}}
    \sum_{n \in S(N - C_1 M_p N^{1/2})}
    \mathbf{1}_{A}\left(n\right)\\
    \geq
    &
    \frac{\delta}{|S(N)|}
    \E_{x \in S(N^{1/2\kappa})}
    \E_{\mathbf{h} \in S(N^{1/4\kappa^{r}})^{\kappa-1}}
    \left(
      \left|A\right| - \left|S(N) \setminus
      S\left(N - C_1 M_p N^{1/2} \right)\right|
    \right)\\
    \geq
    &
    \frac{7}{8}\delta^2.
  \end{aligned}
  \end{equation}
  Using 
  \cref{no structure implies balanced function property}, 
  and Inequalities 
  \eqref{first essential}, \eqref{second essential}, 
  \eqref{third essential}, and \eqref{fourth essential},
  we obtain that
  \begin{align*}
    \E_{n \in S(N)}
    \E_{x \in S(N^{1/2\kappa})}
    \E_{\mathbf{h} \in S(N^{1/4\kappa^{r}})^{\kappa-1}}
    f(n)
    f\left(n + p \left(\sigma_{\kappa-1}(x, \mathbf{h})\right)\right)
    \leq
    - \frac{\delta^2}{2}
  \end{align*}
  holds, which implies that
  \begin{equation*}
    \left\vert 
      \E_{n \in S(N)}
      \E_{x \in S(N^{1/2\kappa})}
      \E_{\mathbf{h} \in S(N^{1/4\kappa^{r}})^{\kappa-1}}
      f(n)
      f\left(n + p \left(\sigma_{\kappa-1}(x, \mathbf{h})\right)\right)
    \right\vert
    \geq
    \frac{\delta^2}{2}
  \end{equation*}
  holds. 
  This proves \cref{avoidance in difference set implies correlation}.
\end{proof}

\subsection{Degree reduction}

In the following lemma, we assume for a subset of $S(N)$ that
the correlation between the balanced function $f$ and the 
composition of $f$ with a shift map by a polynomial \(p(x)\) is large. 
With this assumption, we obtain that there is a polynomial of degree one less than 
that of \(p(x)\),
such that the correlation between $f$ and its composition with 
a shift map by this derived polynomial is also large.
The coefficients of the obtained polynomial are bounded in terms of the 
considered polynomial, and the lower bound of the correlation is analogous 
to that of the prior correlation.

\begin{lemma} [Degree lowering]\label[lemma]{degree reduction of polynomial}
  Let $\kappa$, $r$ be positive integers 
  with $\kappa \geq 2$, and $\delta_0\in (0,1]$. 
  Let $1 \leq m \leq \kappa-1$ be an integer and $p(x) \in \mathcal{O}_K[x]$ be 
  a polynomial of degree $\kappa-m+1$. 
  Let $N$ be a positive integer
  and $A$ be a subset of $S(N)$ of density $\delta \geq \delta_0$.
  Assume that the balanced function $f$ of $A$ satisfies
  \[
    \left\vert 
      \E_{n \in S(N)}
      \E_{x \in S(N^{1/{2\kappa}})}
      \E_{\mathbf{h} \in S(N^{1/4\kappa^{r}})^{\kappa-m}}
      f(n)
      f\left(
        n + p\left(\sigma_{\kappa-m}(x, \mathbf{h})\right)
      \right)
    \right\vert  
    \geq
    \frac{\delta^{2^{m-1} + 1}}{2^{3 \cdot 2^{m-1} - 2}}.
  \]
  If $N \geq N_{m+1}$ holds where
  \[
    N_{m+1} =
    \left\lceil
      \left(
        \frac{2^{3 \cdot 2^{\kappa-1} + d} C_1}{\delta_0^{2^{\kappa - 1} + 1}}
      \right)^{4 \kappa^{r}}
    \right\rceil
    M_p^2,
  \]
  then there exists a polynomial $p^{\prime}(x) \in \mathcal{O}_K[x]$
  of degree $\kappa-m$ with 
  \[
    M_{p'} \leq C_2 M_p N^{1/4 \kappa^{r-1}},
  \]
  and the balanced function $f$ of $A$ 
  satisfies
  \[
    \left\vert 
      \E_{n \in S(N)}
      \E_{x \in S(N^{1/{2\kappa}})}
      \E_{\mathbf{h} \in S(N^{1/4\kappa^{r}})^{\kappa-m-1}}
      f(n)
      f\left(
        n + p^{\prime}\left(\sigma_{\kappa-m-1}(x, \mathbf{h})\right)
      \right)
    \right\vert  
    \geq
    \frac{\delta^{2^{m} + 1}}{2^{3 \cdot 2^{m} - 2}}.
  \]
\end{lemma}

\begin{proof}
  Let the polynomial $p(x)$ be given by 
  $c_0 + c_1 x + \dots + c_{\kappa-m+1} x^{\kappa-m+1}$. 
  Assume that \(N \geq N_{m + 1}\).
  We have 
  \begin{equation*} 
    \left\vert 
      \E_{n \in S(N)}
      \E_{x \in S(N^{1/{2\kappa}})}
      \E_{\mathbf{h} \in S(N^{1/4\kappa^{r}})^{\kappa-m}}
      f(n)
      f\left(
        n + p\left(\sigma_{\kappa-m}(x, \mathbf{h})\right)
      \right)
    \right\vert  
    \geq
    \frac{\delta^{2^{m-1} + 1}}{2^{3 \cdot 2^{m-1} - 2}}.
  \end{equation*}
  Denoting $\E_{n \in S(N)}\E_{x \in S(N^{1/{2\kappa}})}$
  by $\E_{n,x}$, and using the triangle inequality, we get that
  \begin{align*}
    \E_{n,x}
    \left(
      \left\vert f(n)\right\vert
      \cdot
      \left\vert 
        \E_{\mathbf{h}  \in S(N^{1/4\kappa^{r}})^{\kappa-m}}
        f\left(
          n + p\left(\sigma_{\kappa-m}(x, \mathbf{h})\right)
        \right)
      \right\vert
    \right)
    \geq
    \frac{\delta^{2^{m-1} + 1}}{2^{3 \cdot 2^{m-1} - 2}}.
  \end{align*}
  Applying the Cauchy--Schwarz inequality and using 
  \cref{properties of f}, the above inequality gives that
  \begin{equation*} 
    \E_{n,x}
    \left\vert 
      \E_{\mathbf{h}  \in S(N^{1/4\kappa^{r}})^{\kappa-m}}
      f\left(
        n + p\left(\sigma_{\kappa-m}(x, \mathbf{h})\right)
      \right)
    \right\vert^2
    \geq
    \frac{\delta^{2^{m} + 1}}{2^{3 \cdot 2^{m} - 4}},
  \end{equation*}
  which implies that
  \begin{equation} \label{reduction of sigma}
  \begin{aligned}
    \frac{\delta^{2^{m} + 1}}{2^{3 \cdot 2^{m} - 4}}
    &
    \leq
    \E_{n,x}
    \left\vert
      \E_{\mathbf{h}  \in S(N^{1/4\kappa^{r}})^{\kappa-m}} 
      f\left(
        n + p\left(\sigma_{\kappa-m}(x, \mathbf{h})\right)
      \right)
    \right\vert^2\\
    &
    =
    \E_{n,x}
    \left\vert 
      \E_{\mathbf{h}  \in S(N^{1/4\kappa^{r}})^{\kappa-m-1}}
      \E_{k \in S(N^{1/4\kappa^{r}})}
      f 
      \left(
        n +  p\left(\sigma_{\kappa-m-1}(x, \mathbf{h}) + k\right)
      \right)
    \right\vert^2\\
    &
    \leq
    \E_{n,x}
    \E_{\mathbf{h}  \in S(N^{1/4\kappa^{r}})^{\kappa-m-1}}
    \left\vert 
      \E_{k \in S(N^{1/4\kappa^{r}})}
      f 
      \left(
        n +  p\left(\sigma_{\kappa-m-1}(x, \mathbf{h}) + k\right)
      \right)
    \right\vert^2.
  \end{aligned}
  \end{equation}
  Note that given any $n \in S(N)$, $x \in S(N^{1/{2\kappa}})$
  and $\mathbf{h}  \in S(N^{1/4\kappa^{r}})^{\kappa-m-1}$, we have
  \begin{align*}
    &
    \left\vert 
      \sum_{k \in S(N^{1/4\kappa^{r}})}
      f 
      \left(
        n +  p\left(\sigma_{\kappa-m-1}(x, \mathbf{h}) + k\right)
      \right)
    \right\vert^2\\
    =
    &
    \left( \sum_{k \in S(N^{1/4\kappa^{r}})}
    f\left(n +  p\left(\sigma_{\kappa-m-1}(x, \mathbf{h}) + k\right)\right) \right)
    \left( \sum_{k' \in S(N^{1/4\kappa^{r}})}
    f\left(n +  p\left(\sigma_{\kappa-m-1}(x, \mathbf{h}) + k'\right)\right) \right)\\
    =
    &
    \sum_{\substack{
      k \neq k'\\ 
      k,k' \in S(N^{1/4\kappa^{r}}) 
    }}
    f\left(n +  p\left(\sigma_{\kappa-m-1}(x, \mathbf{h}) + k\right)\right)
    f\left(n +  p\left(\sigma_{\kappa-m-1}(x, \mathbf{h}) + k'\right)\right)\\
    &
    +
    \sum_{k \in S(N^{1/4\kappa^{r}})}
    \left(
      f\left(n +  p\left(\sigma_{\kappa-m-1}(x, \mathbf{h}) + k\right)\right)
    \right)^2\\
    \leq
    &
    \sum_{\substack{
      k \neq k'\\ 
      k,k' \in S(N^{1/4\kappa^{r}}) 
    }}
    f\left(n +  p\left(\sigma_{\kappa-m-1}(x, \mathbf{h}) + k\right)\right)
    f\left(n +  p\left(\sigma_{\kappa-m-1}(x, \mathbf{h}) + k'\right)\right)
    +
    \sum_{k \in S(N^{1/4\kappa^{r}})} 1\\ 
    =
    &
    \sum_{\substack{
      k \neq k'\\
      k,k' \in S(N^{1/4\kappa^{r}}) 
    }}
    f\left(n +  p\left(\sigma_{\kappa-m-1}(x, \mathbf{h}) + k\right)\right)
    f\left(n +  p\left(\sigma_{\kappa-m-1}(x, \mathbf{h}) + k'\right)\right)
    +
    |S(N^{1/4\kappa^{r}})|.
  \end{align*}
  Therefore, we obtain that
  \begin{align*}
    \frac{1}{|S(N^{1/4\kappa^{r}})|^2}
    \sum_{\substack{
      k \neq k'\\
      k,k' \in S(N^{1/4\kappa^{r}}) 
    }}
    f\left(n +  p\left(\sigma_{\kappa-m-1}(x, \mathbf{h}) + k\right)\right)
    f\left(n +  p\left(\sigma_{\kappa-m-1}(x, \mathbf{h}) + k'\right)\right)\\
    \geq
    \left\vert 
      \E_{k \in S(N^{1/4\kappa^{r}})}
      f 
      \left(
        n +  p\left(\sigma_{\kappa-m-1}(x, \mathbf{h}) + k\right)
      \right)
    \right\vert^2
    -
    \frac{1}{|S(N^{1/4\kappa^{r}})|}.
  \end{align*}
  Note that 
  $
  N \geq 
  \left(
    \frac{2^{3 \cdot 2^m - 3}}{\delta^{2^m + 1}}
  \right)^{4 \kappa^r / d}
  $
  implies that
  \[
    \frac{1}{|S(N^{1/4\kappa^{r}})|} 
    \leq 
    \frac{\delta^{2^{m} + 1}}{2^{3 \cdot 2^{m} - 3}}.
  \]
  Therefore, taking expectations over $n,x$ and $\mathbf{h}$ in their respective 
  ranges,
  Inequality \eqref{reduction of sigma} implies
  \begin{align*}
    \sum_{\substack{
      k \neq k'\\ 
      k,k' \in S(N^{1/4\kappa^{r}}) 
    }}
    \E_{n,x}
    \E_{\mathbf{h} \in S(N^{1/4\kappa^{r}})^{\kappa-m-1}}
    f\left(n +  p\left(\sigma_{\kappa-m-1}(x, \mathbf{h}) + k\right)\right)
    f\left(n +  p\left(\sigma_{\kappa-m-1}(x, \mathbf{h}) + k'\right)\right)\\
    \geq
    \frac{\delta^{2^{m} + 1}}{2^{3 \cdot 2^{m} - 3}} 
    |S(N^{1/4\kappa^{r}})|^2.
  \end{align*}
  The pigeonhole principle implies that 
  there exist distinct elements $k$ and $k'$ in the set $S(N^{1/4\kappa^{r}})$ 
  satisfying 
  \begin{align*}
    \E_{n,x}
    \E_{\mathbf{h} \in S(N^{1/4\kappa^{r}})^{\kappa-m-1}}
    f\left(n +  p\left(\sigma_{\kappa-m-1}(x, \mathbf{h}) + k\right)\right)
    f\left(n +  p\left(\sigma_{\kappa-m-1}(x, \mathbf{h}) + k'\right)\right)
    \geq
    \frac{\delta^{2^{m} + 1}}{2^{3 \cdot 2^{m} - 3}}.
  \end{align*}  
  Note that for every nonnegative integer $j \leq \kappa-1$ and 
  $(x,\mathbf{h})$ in $S(N^{1/2\kappa}) \times S(N^{1/4\kappa^{r}})^{j}$, 
  the polynomial $p(x)$ satisfies
  \[
    p\left(\sigma_{j}(x, \mathbf{h})\right)
    \in 
    S(C_1 M_p N^{1/2}),
  \]
  and consequently, the sets $S(N)$ and 
  $S(N) + p\left(\sigma_{j}(x,\mathbf{h})\right)$ 
  are contained in the set
  \[
    T_N \coloneq 
    S(N + C_1 M_p N^{1/2}).
  \]
  Note that the function $f$ vanishes outside $S(N)$.
  Thus, for any element $(x,\mathbf{h})$ of the set 
  $S(N^{1/{2\kappa}}) \times S(N^{1/4\kappa^{r}})^{ d - m - 1}$, 
  we observe that
  \begin{align*}
    &
    \E_{n \in S(N)}
    f(n)
    f\left(
      n +  p\left(\sigma_{\kappa-m-1}(x, \mathbf{h}) + k'\right)
      - p\left(\sigma_{\kappa-m-1}(x, \mathbf{h}) + k\right)
    \right)\\
    =
    &
    \frac{1}{|S(N)|}
    \sum_{n \in S(N)}
    f(n)
    f\left(
      n +  p\left(\sigma_{\kappa-m-1}(x, \mathbf{h}) + k'\right)
      - p\left(\sigma_{\kappa-m-1}(x, \mathbf{h}) + k\right)
    \right)\\
    =
    &
    \frac{1}{|S(N)|}
    \sum_{n \in T_N}
    f(n)
    f\left(
      n +  p\left(\sigma_{\kappa-m-1}(x, \mathbf{h}) + k'\right)
      - p\left(\sigma_{\kappa-m-1}(x, \mathbf{h}) + k\right)
    \right)\\
    =
    &
    \frac{1}{|S(N)|}
    \Biggl(
      \sum_{n \in S(N) + p(\sigma_{\kappa-m-1}(x, \mathbf{h}) + k)}
      f(n)
      f\left(
        n +  p\left(\sigma_{\kappa-m-1}(x, \mathbf{h}) + k'\right)
        - p\left(\sigma_{\kappa-m-1}(x, \mathbf{h}) + k\right)
      \right)
      \\
      &
      +
      \sum_{n \in T_N
      \setminus \left(S(N) + p(\sigma_{\kappa-m-1}(x, \mathbf{h}) + k)\right)}
      f(n)
      f\left(
        n +  p\left(\sigma_{\kappa-m-1}(x, \mathbf{h}) + k'\right)
        - p\left(\sigma_{\kappa-m-1}(x, \mathbf{h}) + k\right)
      \right)
    \Biggr)\\
    \geq
    &
    \frac{1}{|S(N)|}
    \sum_{n \in S(N)}
    f\left(n +  p\left(\sigma_{\kappa-m-1}(x, \mathbf{h}) + k\right)\right)
    f\left(n +  p\left(\sigma_{\kappa-m-1}(x, \mathbf{h}) + k'\right)\right)\\
    &
    - 
    \frac{|S(N + C_1 M_p N^{1/2})|-|S(N)|}{|S(N)|}.
  \end{align*}
  Note that
  $
    N \geq 
    \left(
      \frac{2^{3 \cdot 2^m + d} C_1 M_p}
      {\delta^{2^{m} + 1}}
    \right)^2,
  $
  implies that
  \[
    \frac{|S(N + C_1 M_p N^{1/2})|-|S(N)|}{|S(N)|}
    \leq 
    \frac{\delta^{2^{m} + 1}}{2^{3 \cdot 2^{m} - 2}}.
  \]
  Thus, taking expectations over $x$ and $\mathbf{h}$
  in their respective ranges, the above inequality yields 
  \begin{equation} \label{1 degree reduced}
    \E_{n,x}
    \E_{\mathbf{h} \in S(N^{1/4\kappa^{r}})^{d - m - 1}}
    f(n)
    f\left(
      n +  p\left(\sigma_{\kappa-m-1}(x, \mathbf{h}) + k'\right)
      - p\left(\sigma_{\kappa-m-1}(x, \mathbf{h}) + k\right)
    \right)
    \geq
    \frac{\delta^{2^{m} + 1}}{2^{3 \cdot 2^{m} - 2}}.
  \end{equation}

  Using the binomial theorem, we observe that
  \begin{align*}
    &
    p\left(\sigma_{\kappa-m-1}(x, \mathbf{h}) + k'\right) -
    p\left(\sigma_{\kappa-m-1}(x, \mathbf{h}) + k\right)\\
    =
    &
    \sum_{i=0}^{\kappa-m+1} c_i 
    \left(
      \left(\sigma_{\kappa-m-1}(x, \mathbf{h}) + k'\right)^i -
      \left(\sigma_{\kappa-m-1}(x, \mathbf{h}) + k\right)^i
    \right)\\
    =
    &
    \sum_{i=1}^{\kappa-m+1} c_i 
    \left(
      \sum_{j=0}^{i-1}
      {\binom{i}{j}}
      \left(k'^{i-j} - k^{i-j}\right)
      \left(\sigma_{\kappa-m-1}(x, \mathbf{h})\right)^{j}
    \right)\\
    =
    &
    \sum_{j=0}^{\kappa-m} \left(\sigma_{\kappa-m-1}(x, \mathbf{h})\right)^{j}
    \left(
      \sum_{i=j+1}^{\kappa-m+1} c_i 
      {\binom{i}{j}}
      \left(k'^{i-j} - k^{i-j}\right)
    \right).
  \end{align*}
  For every $0 \leq j \leq \kappa-m$, let us denote 
  \[
    b_j
    \coloneq
    \sum_{i=j+1}^{\kappa-m+1} c_i
    {\binom{i}{j}}
    \left(k'^{i-j} - k^{i-j}\right),
  \]
  and note that $b_{\kappa-m}$ is nonzero.
  Note that for every $1 \leq i \leq \kappa-m+1$ and $0 \leq j \leq \kappa-m$,
  we have $c_i \in S(M_p)$ and 
  $k'^{i-j} - k^{i-j} \in S(2 d^{2(\kappa-1)} E^{\kappa-1} N^{1/4\kappa^{r-1}})$.
  Hence, for every $0 \leq j \leq d-m$, we obtain that
  \[
    b_j \in S(C_2 M_p N^{1/4\kappa^{r-1}}).
  \]

  Let us set the polynomial 
  $p^{\prime}(x) \coloneq b_0 + b_1 x + \dots + b_{\kappa-m} x^{\kappa-m}$,
  and observe that $p'(x) \in \mathcal{O}_K[x]$ is a polynomial 
  of degree $\kappa-m$, 
  such that if $N \geq N_{m+1}$, then we have
  \[
    M_{p'} \leq C_2 M_p N^{1/4\kappa^{r-1}},
  \]
  and Inequality \eqref{1 degree reduced}
  implies that
  \begin{equation} \label{(m+1)th iteration}
    \left\vert 
      \E_{n \in S(N)}
      \E_{x \in S(N^{1/{2\kappa}})}
      \E_{\mathbf{h} \in S(N^{1/4\kappa^{r}})^{\kappa-m-1}}
      f(n)
      f\left(
        n + p^{\prime}\left(\sigma_{\kappa-m-1}(x, \mathbf{h})\right)
      \right)
    \right\vert
    \geq
    \frac{\delta^{2^{m} + 1}}{2^{3 \cdot 2^{m} - 2}}
  \end{equation}
  holds. This proves \cref{degree reduction of polynomial}.
\end{proof}

Applying the above result iteratively $\kappa-m$ times to a polynomial 
of degree $\kappa-m+1$ gives the following result.

\begin{proposition} [Linearization]
  \label[proposition]{Linearization}
  Let $\kappa$, $r$ be positive integers with $r \geq 2$, and $\delta_0 \in (0,1]$.
  Let $1 \leq m \leq \kappa$ be an integer and
  $p_m(x) \in \mathcal{O}_K[x]$ be a polynomial of degree $\kappa-m+1$.
  Let $N$ be positive integer
  and $A$ be a subset of $S(N)$ of density $\delta \geq \delta_0$.
  Assume that the balanced function $f$ of $A$ satisfies
  \[
    \left\vert 
      \E_{n \in S(N)}
    \E_{x \in S(N^{1/{2\kappa}})}
    \E_{\mathbf{h} \in S(N^{1/4\kappa^{r}})^{\kappa-m}}
      f(n)
      f\left(
        n + p_m\left(\sigma_{\kappa-m}(x, \mathbf{h})\right)
      \right)
    \right\vert  
    \geq
    \frac{\delta^{2^{m-1} + 1}}{2^{3 \cdot 2^{m-1} - 2}}.
  \]
  If $N \geq N_0$ holds where 
  \[
    N_{0} 
    =
    \left\lceil
      \left(
        \frac{2^{3 \cdot 2^{\kappa-1} + d + 1} C_1 C_2}
        {\delta_0^{2^{\kappa - 1} + 1}}
      \right)^{8\kappa^{r}}
    \right\rceil
    M_{p_m}^4,
  \]
  then there exists a linear polynomial $p_\kappa(x)$ with
  \[
    M_{p_\kappa} \leq C_2^{\kappa} M_{p_m} N^{1/4\kappa^{r-2}},
  \] 
  and the balanced function $f$ of $A$ satisfies 
  \[
    \left\vert 
      \E_{n \in S(N)}
      \E_{x \in S(N^{1/{2\kappa}})}
      f(n)
      f \left(n + p_\kappa(x) \right)
    \right\vert  
    \geq
    \frac{\delta^{2^{\kappa-1} + 1}}{2^{3 \cdot 2^{\kappa-1} - 2}}.
  \]
\end{proposition}

\begin{proof}
  Note that if $m = \kappa$, in particular, if $\kappa = 1$, 
  then we are done. 
  Let us take $\kappa \geq 2$ and $1 \leq m \leq \kappa-1$.
  Assume that $N \geq N_0$ holds. 
  Then
  \begin{equation} \label{mth assertion}
    \left\vert 
      \E_{n \in S(N)}
      \E_{x \in S(N^{1/{2\kappa}})}
      \E_{\mathbf{h} \in S(N^{1/4\kappa^{r}})^{\kappa-m}}
      f(n)
      f\left(
        n + p_m\left(\sigma_{\kappa-m}(x, \mathbf{h})\right)
      \right)
    \right\vert  
    \geq
    \frac{\delta^{2^{m-1} + 1}}{2^{3 \cdot 2^{m-1} - 2}}
  \end{equation}
  holds.

  Since 
  \(
    N \geq
    \left\lceil
      \left(
        \frac{2^{3 \cdot 2^{\kappa-1} + d} C_1}{\delta_0^{2^{\kappa - 1} + 1}}
      \right)^{4 \kappa^{r}}
    \right\rceil
    M_{p_m}^2
  \)
  holds, 
  applying \cref{degree reduction of polynomial}, 
  we obtain a polynomial $p_{m+1}(x) \in \mathcal{O}_K[x]$
  of degree $\kappa-m$ such that 
  \[
    M_{p_{m+1}} \leq C_2 M_{p_m} N^{1/4\kappa^{r-1}}
  \]
  holds and the balanced function $f$ of $A$ 
  satisfies
  \[
    \left\vert 
      \E_{n \in S(N)}
      \E_{x \in S(N^{1/{2\kappa}})}
      \E_{\mathbf{h} \in S(N^{1/4\kappa^{r}})^{\kappa-m-1}}
      f(n)
      f\left(
        n + p_{m+1}\left(\sigma_{\kappa-m-1}(x, \mathbf{h})\right)
      \right)
    \right\vert  
    \geq
    \frac{\delta^{2^{m} + 1}}{2^{3 \cdot 2^{m} - 2}}.
  \]
  We claim that \cref{degree reduction of polynomial} can be applied 
  $\kappa - m$ times iteratively.
  To prove it, 
  let us assume that \cref{degree reduction of polynomial} can be applied 
  $i$ times iteratively,
  for some $1 \leq i < \kappa - m$.
  Therefore, for every $1 \leq j \leq i$, we have 
  $N \geq N_{m+j}$
  where
  \[
    N_{m+j} \coloneq
    \left\lceil
      \left(
        \frac{2^{3 \cdot 2^{\kappa-1} + d} C_1}{\delta_0^{2^{\kappa - 1} + 1}}
      \right)^{4 \kappa^{r}}
    \right\rceil
    M_{p_{m + j - 1}}^2,
  \]
  and there exists a 
  polynomial $p_{m + j}(x)$ of degree $(\kappa - m + 1) - j$ with 
  \[
    M_{p_{m + j}} \leq C_2 M_{p_{m + j - 1}} N^{1/4\kappa^{r - 1}},
  \]
  and the balanced function \(f\) of the subset \(A\) of \(S(N)\) 
  satisfies
  \[
    \left\vert 
      \E_{n \in S(N)}
      \E_{x \in S(N^{1/{2\kappa}})}
      \E_{\mathbf{h} \in S(N^{1/4\kappa^{r}})^{\kappa-m-j}}
      f(n)
      f\left(
        n + p_{m+j}\left(\sigma_{\kappa-m-j}(x, \mathbf{h})\right)
      \right)
    \right\vert   
    \geq
    \frac{\delta^{2^{m+j-1} + 1}}{2^{3 \cdot 2^{m+j-1} - 2}}.
  \]

  We claim that $N \geq N_{m+i+1}$ where 
  \[
    N_{m + i + 1}
    \coloneq
    \left\lceil
      \left(
        \frac{2^{3 \cdot 2^{\kappa-1} + d} C_1}{\delta_0^{2^{\kappa - 1} + 1}}
      \right)^{4 \kappa^{r}}
    \right\rceil
    M_{p_{m+i}}^2.
  \]
  Note that 
  \[
    M_{p_{m + i}} \leq C_2^{i} M_{p_{m}} N^{i/4\kappa^{r - 1}}
    \leq C_2^\kappa M_{p_{m}} N^{1/4\kappa^{r - 2}}
  \]
  holds, which yields that 
  \begin{align*}
    \left\lceil
      \left(
        \frac{2^{3 \cdot 2^{\kappa-1} + d} C_1}{\delta_0^{2^{\kappa - 1} + 1}}
      \right)^{4 \kappa^{r}}
    \right\rceil
    M_{p_{m+i}}^2
    & \leq 
    2
    \left(
      \frac{2^{3 \cdot 2^{\kappa-1} + d} C_1}{\delta_0^{2^{\kappa - 1} + 1}}
    \right)^{4 \kappa^{r}}
    C_2^{2\kappa} M_{p_{m}}^2 N^{1/2\kappa^{r - 2}}
    \\
    & \leq
    \left(
      \frac{2^{3 \cdot 2^{\kappa-1} + d + 1} C_1 C_2}
      {\delta_0^{2^{\kappa - 1} + 1}}
    \right)^{4 \kappa^{r}}
    M_{p_{m}}^2 N^{1/2\kappa^{r - 2}} \\
    & \leq N_0^{1/2} N^{1/2} \\
    & \leq N
  \end{align*}
  holds.
  This gives $N \geq N_{m + i + 1}$, 
  and hence, 
  the claim that \cref{degree reduction of polynomial}
  can be applied iteratively \(\kappa - m\) times follows.
  This yields a linear polynomial $p_\kappa(x) = \alpha x + \beta$ 
  with $\alpha \neq 0$ and 
  \[
    M_{p_\kappa}
    \leq 
    C_2^{\kappa} M_{p_m} N^{1/4\kappa^{r-2}},
  \] 
  such that 
  the balanced function $f$ of the subset $A$ of $S(N)$ 
  satisfies 
  \[
    \left\vert 
      \E_{n \in S(N)}
      \E_{x \in S(N^{1/{2\kappa}})}
      f(n)
      f\left(
        n + \alpha x + \beta
      \right)
    \right\vert  
    \geq
    \frac{\delta^{2^{\kappa-1} + 1}}{2^{3 \cdot 2^{\kappa-1} - 2}}.
  \]
  Hence, \cref{Linearization} is proved.
\end{proof}

The above result for $m=1$ gives the following result, which is more convenient 
to use in the proof of \cref{density increment lemma}.
In particular, the following corollary provides the conclusion after applying
\cref{degree reduction of polynomial} iteratively, $\kappa-1$ times.

\begin{corollary} \label[corollary]{linearization corollary}
  Let $r \geq 2$ be an integer and $\delta_0\in (0,1]$.
  Let $p(x) \in \mathcal{O}_K[x]$ be a nonconstant polynomial of degree $\kappa$.
  Let $N$ be a positive integer
  and $A$ be a subset of $S(N)$ of density $\delta \geq \delta_0$.
  Assume that the balanced function $f$ of $A$ satisfies
  \[
    \left\vert 
      \E_{n \in S(N)}
      \E_{x \in S(N^{1/{2\kappa}})}
      \E_{\mathbf{h} \in S(N^{1/4\kappa^{r}})^{\kappa-1}}
      f(n)
      f\left(
        n + p\left(\sigma_{\kappa-1}(x, \mathbf{h})\right)
      \right)
    \right\vert  
    \geq
    \frac{\delta^{2}}{2}.
  \]
  If $N \geq N_0$ where 
  \[
    N_0
    = 
    \left\lceil
      \left(
        \frac{2^{3 \cdot 2^{\kappa-1} + d + 1} C_1 C_2}
        {\delta_0^{2^{\kappa - 1} + 1}}
      \right)^{8\kappa^{r}}
    \right\rceil
    M_{p}^4,
  \]
  then there exist elements
  $\alpha, \beta$ in \(S( C_2^{\kappa} M_{p} N^{1/4\kappa^{r-2}} )\) 
  satisfying \(\alpha \neq 0\) and 
  \[
    \left\vert 
      \E_{n \in S(N)}
      \E_{x \in S(N^{1/{2\kappa}})}
      f(n)
      f \left(n + \alpha x +\beta \right)
    \right\vert  
    \geq
    \frac{\delta^{2^{\kappa-1} + 1}}{2^{3 \cdot 2^{\kappa-1} - 2}}.
  \]
\end{corollary}

\subsection{Density increment lemma}

We establish the following density increment lemma using the above results.
In this lemma, we consider a set such that its difference set avoids the
nonzero values of a polynomial of degree $\kappa$. 
We apply \cref{avoidance in difference set implies correlation} and
\cref{linearization corollary}, to obtain a lower bound
for the correlation between the balanced function 
$f$ and its composition with a shift by a linear polynomial.
By an application of the pigeonhole principle, we obtain the desired result.

\begin{lemma}[Density increment] \label[lemma]{density increment lemma}
  Let $r \geq 2$ be an integer, $p(x)$ be a nonconstant polynomial 
  of degree $\kappa$ in $\mathcal{O}_K[x]$, and $\delta_0 \in (0,1]$.  
  Let $N$ be a positive integer and $A$ be a subset of $S(N)$ of density
  $\delta \geq \delta_0$.
  Suppose the difference set $A - A$ does not contain the nonzero values of $p(x)$.
  If $N \geq N_{\delta_0}$ where 
  \[
    N_{\delta_0}
    =
    \left\lceil
      \left(
        \frac{2^{3 \cdot 2^{\kappa-1} + d + 5} C_1 C_2}
        {\delta_0^{2^{\kappa - 1} + 1}}
      \right)^{8\kappa^{r}}
    \right\rceil
    M_{p}^4,
  \]
  then there exist elements $\alpha$ in
  $S ( C_2^{\kappa} M_{p} N^{1/4\kappa^{r-2}} )$
  and $n$ in $\mathcal{O}_K$ 
  such that the set $L \coloneq n + \alpha S(N^{1/2\kappa})$ is contained in $S(N)$ and 
  the density of $A$ in $L$
  is at least $\delta + c \delta^C$,
  where 
  $c = 1/2^{3 \cdot 2^{\kappa-1} + 2}$ and $C = 2^{\kappa-1} + 1$.
\end{lemma}

\begin{proof}
  Let $N_{\delta_0}$ be a positive integer as mentioned above
  and $N \geq N_{\delta_0}$ holds.
  Note that 
  $N
    \geq
    \left(
      \frac{2^{d + 5} C_1 M_p}
      {\delta_0^{2}}
    \right)^{2}
  $ holds. Therefore, \cref{avoidance in difference set implies correlation}
  yields that
  \begin{equation} \label{eq:key inequality}
    \left\vert 
      \E_{n \in S(N)}
      \E_{x \in S(N^{1/2\kappa})}
      \E_{\mathbf{h} \in S(N^{1/4\kappa^{r}})^{\kappa-1}}
      f(n)
      f\left(n + p \left(\sigma_{\kappa-1}(x, \mathbf{h})\right)\right)
    \right\vert
    \geq
    \frac{\delta^2}{2}.
  \end{equation}
  Since 
  \[
    N \geq
    \left\lceil
      \left(
        \frac{2^{3 \cdot 2^{\kappa-1} + d + 1} C_1 C_2}
        {\delta_0^{2^{\kappa - 1} + 1}}
      \right)^{8\kappa^{r}}
    \right\rceil
    M_{p}^4
  \]
  holds, \cref{linearization corollary} implies that
  we get a linear polynomial $\alpha x + \beta$ in $\mathcal{O}_K[x]$
  with $\alpha \neq 0$, and 
  $\alpha, \beta$ lying in the set
  $S(C_2^\kappa M_{p} N^{1/4\kappa^{r-2}})$, and satisfying
  \[
    \left\vert 
      \E_{n \in S(N)}
      \E_{x \in S(N^{1/2\kappa})}
      f(n)
      f(n + \alpha x + \beta)
    \right\vert
    \geq
    \frac{\delta^{2^{\kappa-1}+1}}{2^{3 \cdot 2^{\kappa-1} - 2}}.
  \] 
  The triangle inequality implies that 
  \begin{equation} \label{dth assertion}
    \E_{n \in S(N)}
    \left\vert 
      \E_{x \in S(N^{1/2\kappa})}
      f(n)
      f(n + \alpha x + \beta)
    \right\vert
    \geq
    \frac{\delta^{2^{\kappa-1}+1}}{2^{3 \cdot 2^{\kappa-1} - 2}},
  \end{equation}
  Observe that the set
  \[
    U_N \coloneq 
    S(N + C_2^\kappa M_{p} N^{1/4\kappa^{r-2}})
  \]
  contains the sets $S(N)$ and $S(N) - \beta$, and this shows that 
  \begin{align*}
    &
    \E_{n \in S(N)}
    \left| 
      \E_{x \in S(N^{1/2\kappa})}
      f\left(n + \alpha x + \beta\right)
    \right|\\
    =
    &
    \frac{1}{|S(N)|}
    \sum_{n \in S(N)}
    \left\vert 
      \E_{x \in S(N^{1/2\kappa})}
      f(n + \alpha x + \beta)
    \right\vert\\
    \leq
    &
    \frac{1}{|S(N)|}
    \sum_{n \in U_N}
    \left\vert 
      \E_{x \in S(N^{1/2\kappa})}
      f(n + \alpha x + \beta)
    \right\vert\\
    =
    &
    \frac{1}{|S(N)|}
    \sum_{n \in S(N) - \beta}
    \left\vert 
      \E_{x \in S(N^{1/2\kappa})}
      f(n + \alpha x + \beta)
    \right\vert\\
    & \quad
    +
    \frac{1}{|S(N)|}
    \sum_{n \in U_N \setminus (S(N) - \beta)}
    \left\vert 
      \E_{x \in S(N^{1/2\kappa})}
      f(n + \alpha x + \beta)
    \right\vert\\
    \leq
    &
    \frac{1}{|S(N)|}
    \sum_{n \in S(N)}
    \left\vert 
      \E_{x \in S(N^{1/2\kappa})}
      f(n + \alpha x)
    \right\vert
    +
    \frac{1}{|S(N)|}
    \sum_{n \in U_N \setminus (S(N) - \beta)}
    1\\
    \leq
    &
    \E_{n \in S(N)}
    \left\vert 
      \E_{x \in S(N^{1/2\kappa})}
      f\left(n + \alpha x\right)
    \right\vert
    +
    \frac{|S(N + C_2^\kappa M_{p} N^{1/4\kappa^{r-2}})| - |S(N)|}{|S(N)|}
  \end{align*}
  holds. 
  Note that 
  \[
    N \geq
    \left(
      \frac{2^{3 \cdot 2^{\kappa - 1} +d +1} C_2^\kappa M_p}
      {\delta^{2^{\kappa-1} + 1}}
    \right)^2
  \]
  holds, 
  and hence, combining the above estimate with \eqref{dth assertion},
  we obtain that 
  \begin{equation} \label{reduced in linear and free of constant term}
    \E_{n \in S(N)}
    \left\vert 
      \E_{x \in S(N^{1/2\kappa})}
      f(n + \alpha x)
    \right\vert
    \geq
    \frac{\delta^{2^{\kappa-1}+1}}{2^{3 \cdot 2^{\kappa-1} - 1}}
  \end{equation}
  holds. Also, observe that given any element $x$ in $S(N^{1/2\kappa})$,
  the element $\alpha x$ lies in the set $S(d^2 E C_2^\kappa M_{p} N^{3/4})$.
  As a consequence, the set 
  \[
    V_N \coloneq
    S(N + d^2 E C_2^\kappa M_{p} N^{3/4})
  \]
  contains the sets $S(N)$ and $S(N) - \alpha x$.
  This yields that
  \begin{align*}
    &
    \E_{n \in S(N)}
    f\left(n + \alpha x\right)\\
    =
    &
    \frac{1}{|S(N)|}
    \left(
      \sum_{n \in V_N}
      f\left(n + \alpha x\right)
      -
      \sum_{n \in V_N \setminus S(N)}
      f\left(n + \alpha x\right)
    \right)\\
    =
    &
    \frac{1}{|S(N)|}
    \left(
      \sum_{n \in S(N) - \alpha x}
      f\left(n + \alpha x\right)
      +
      \sum_{n \in V_N \setminus (S(N) - \alpha x)}
      f\left(n + \alpha x\right)
      -
      \sum_{n \in V_N \setminus S(N)}
      f\left(n + \alpha x\right)
    \right).
  \end{align*}
  Using \cref{properties of f}, the above yields that
  \begin{align*}
    \left| 
      \E_{n \in S(N)}
      f\left(n + \alpha x\right)
    \right|
    &
    \leq
    \frac{2 \left|V_N \setminus S(N)\right|}{|S(N)|}.
  \end{align*}
  Since 
  \[
    N \geq
    \left(
      \frac{2^{3 \cdot 2^{\kappa - 1} +d +3} d^2 E C_2^\kappa M_p}
      {\delta^{2^{\kappa-1} + 1}}
    \right)^4,
  \]
  therefore, taking expectation 
  over $x$, we obtain that
  \[
    - \frac{\delta^{2^{\kappa-1}+1}}{2^{3 \cdot 2^{\kappa-1}}}
    \leq
    \E_{n \in S(N)}
    \E_{x \in S(N^{1/2\kappa})}
    f(n + \alpha x)
    \leq
    \frac{\delta^{2^{\kappa-1}+1}}{2^{3 \cdot 2^{\kappa-1}}}.
  \]
  Combining this with \eqref{reduced in linear and free of constant term} 
  yields that
  \[
    \frac{\delta^{2^{\kappa-1}+1}}{2^{3 \cdot 2^{\kappa-1} + 1}}
    \leq
    \E_{n \in S(N)}
    \max
    \left\{
      \E_{x \in S(N^{1/2\kappa})}
      f(n + \alpha x), 0
    \right\}.
  \]
  
  Consider the set
  \[
    W_N \coloneq
    S(N - d^2 E C_2^\kappa M_{p} N^{3/4}),
  \]
  and observe that given any $n \in W_N$ and $x \in S(N^{1/2\kappa})$, 
  the element $n + \alpha x$ lies in $S(N)$. Thus, it follows that
  \begin{align*}
  &
    \E_{n \in S(N)}
    \max 
    \left\{
      \E_{x \in S(N^{1/2\kappa})}
      f(n + \alpha x),
      0
    \right\}\\
    &
    \leq
    \frac{1}{|S(N)|}
    \left(
      \sum_{n \in W_N}
      \max 
      \left\{
        \E_{x \in S(N^{1/2\kappa})}
        f(n + \alpha x),
        0
      \right\}
      +
      \sum_{n \in S(N) \setminus W_N}
      1
    \right) \\
    &
    \leq
    \frac{1}{|S(N)|}
    \sum_{n \in W_N}
    \max 
    \left\{
      \E_{x \in S(N^{1/2\kappa})}
      f(n + \alpha x),
      0
    \right\}
    +
    \frac{|S(N) \setminus W_N|}{|S(N)|}.
  \end{align*}
  We have
  \[
    N \geq
    \left(
      \frac{2^{3 \cdot 2^{\kappa - 1} +d +4} d^2 E C_2^\kappa M_p}
      {\delta^{2^{\kappa-1} + 1}}
    \right)^4,
  \]
  and this yields 
  \begin{equation*}
    \frac{1}{|S(N)|}
    \sum_{n \in W_N}
    \max
    \left\{
      \E_{x \in S(N^{1/2\kappa})}
      f(n + \alpha x),
      0
    \right\}
    \geq
    \frac{\delta^{2^{\kappa-1} + 1}}{2^{3 \cdot 2^{\kappa-1} + 2}}
  \end{equation*}
  holds.
  By the pigeonhole principle, we obtain an element $n$ in $W_N$
  such that
  \[
    \E_{x \in S(N^{1/2\kappa})}
    f(n + \alpha x)
    \geq
    \frac{\delta^{2^{\kappa-1}+1}}{2^{3 \cdot 2^{\kappa-1} + 2}}
  \]
  holds. 
  Consider the set
  \[
    L \coloneq \{n + \alpha x : x \in S(N^{1/2\kappa})\} 
    = n + \alpha S(N^{1/2\kappa}),
  \]
  and observe that $L \subseteq S(N)$.
  The above inequality yields that the density of $A$ in $L$ 
  is at least $\delta + c \delta^C$.
  This completes the proof of \cref{density increment lemma}.
\end{proof}

\section{Intersective polynomials}
\label{section:intersective polynomials}

In this section, we define intersective polynomials over the 
ring $\mathcal{O}_K$. 
For a given intersective polynomial $q(x)$ and 
a nonzero element $\alpha$ of $\mathcal{O}_K$, 
we define an auxiliary polynomial $q_{\alpha}(x)$ and 
show that it is an intersective polynomial in $\mathcal{O}_K[x]$.
In \cref{Lucier}, 
we show that if $A$ is a subset of $\mathcal{O}_K$ and 
$A - A$ avoids the 
nonzero values of $q(x)$, then there exists a subset $A'$ of 
$\mathcal{O}_K$, defined in terms of $A$ and $\alpha$, 
such that the difference set 
$A' - A'$ avoids the nonzero values 
of the auxiliary polynomial $q_{\alpha}(x)$.

\begin{definition} \label[definition]{definition of intersective polynomial}
  A polynomial $q(x) \in \mathcal{O}_K[x]$ is said to be \emph{intersective} 
  if for every nonzero element $\alpha$ of $\mathcal{O}_K$, 
  there exists an element $z$ in $\mathcal{O}_K$ 
  such that $\alpha$ divides $q(z)$.
\end{definition}

Note that 
a polynomial in $\mathbb{Z}[x]$ is called an \emph{intersective} polynomial
if it admits a root modulo every positive integer.
Moreover, a polynomial in $\mathbb{Z}[x]$ is intersective
if and only if it has a root in the ring of $p$-adic integers $\mathbb{Z}_p$ 
for every nonzero prime $p$ 
(see \cite[p.~105]{NeukirchANT}, for instance).
Similar to the case of the intersective polynomials over $\mathbb{Z}$, 
it can be proved by a repeated application of the pigeonhole principle that 
a polynomial over $\mathcal{O}_K$ is intersective if and only if 
it admits a root in 
the $\mathfrak{p}$-adic completion $\widehat{\mathcal{O}}_{K,\mathfrak{p}}$ 
of $\mathcal{O}_K$
for every nonzero prime ideal $\mathfrak{p}$ of $\mathcal{O}_K$.

\begin{definition}
  For a nonzero prime ideal $\mathfrak{p}$ of $\mathcal{O}_K$,
  the \emph{$\mathfrak{p}$-adic completion of $\mathcal{O}_K$}
  is denoted by $\widehat{\mathcal{O}}_{K, \mathfrak{p}}$, 
  and is defined to be
  \[
    \widehat{\mathcal{O}}_{K, \mathfrak{p}} \coloneq
    \varprojlim_{n} \frac{\mathcal{O}_K}{\mathfrak{p}^n},
  \]
  where the transition maps are the projection maps.
\end{definition}

In the following, 
\(q(x)\) denotes an intersective polynomial 
of degree $\kappa$ over the ring $\mathcal{O}_K$, 
\(\alpha\) denotes a 
nonzero element of $\mathcal{O}_K$,
and $\mathfrak{a}$ denotes the ideal of \(\mathcal{O}_K\) generated by $\alpha$.
For every nonzero prime ideal $\mathfrak{p}$ of $\mathcal{O}_K$, we
fix a root $z_{\mathfrak{p}}\coloneq \{s_{\mathfrak{p}^n}\}_{n \geq 1}$ of $q(x)$ 
in $\widehat{\mathcal{O}}_{K, \mathfrak{p}}$. 

Write $\alpha = a_1 e_1 + \dots + a_{d} e_{d}$, 
and for every $1 \leq i \leq d$ and $1 \leq j \leq d$,
write 
\[
  e_i e_j = \sum_{k=1}^{d} c_{i j k} e_k,
\]
where $a_i$ and $c_{ijk}$ are integers.
Consider the map $T_{\alpha} \colon \mathcal{O}_K \to \mathcal{O}_K$
defined as $T_{\alpha}(x) = \alpha x$ for every $x$ in $\mathcal{O}_K$,
and note that $T_{\alpha}$ is a $\mathbb{Z}$-linear endomorphism on 
$\mathcal{O}_K$.
The matrix corresponding to $T_{\alpha}$
with respect to the basis $\{e_1, \dots, e_{d}\}$ is given by
\[
  \mathcal{A} = 
  \left[
    a_{j k}
  \right]_{d \times d},
  \quad \text{where} \quad
  a_{j k} =
  \sum_{i=1}^{d} a_i c_{i k j}.
\]
Let us consider the characteristic polynomial $\chi_{\mathcal{A}}(x)$ 
of the matrix \(\mathcal{A}\) given by 
\[
  \chi_{\mathcal{A}}(x) = \det(x I - \mathcal{A}).
\]
Note that $\chi_{\mathcal{A}}(x)$ has integer coefficients,
which have absolute value
at most $(d^2 E M_{\alpha})^d$.
Let \(f(x)\) denote the polynomial with integer coefficients
satisfying
\[
  \chi_{\mathcal{A}}(x) = x f(x) + (-1)^d \det(\mathcal{A}).
\]
Note that $\chi_{\mathcal{A}}(\alpha) = 0$, and this yields that
\begin{equation} \label{expression of alpha}
  \frac{1}{\alpha}
  =
  (-1)^{d-1} \frac{f(\alpha)}{\det(\mathcal{A})}.
\end{equation}

\begin{proposition} \label[proposition]{existence of r_a}
  There exists an element $r_\alpha$ in the set 
  $S\left((d^2 E M_{\alpha})^{d}\right)$ such that 
  the element $q(r_\alpha)$ lies in $\mathfrak{a}$, and 
  for any prime power ideal $\mathfrak{p}^n$ of $\mathcal{O}_K$ 
  containing $\mathfrak{a}$, the element $r_\alpha$ 
  goes to $s_{\mathfrak{p}^n}$ 
  under the projection map from $\mathcal{O}_K$
  to $\mathcal{O}_K / \mathfrak{p}^n$.
\end{proposition}

\begin{proof}
  Let 
  \(
    \mathfrak{p}_1^{v_1} \dots \mathfrak{p}_s^{v_s}
  \)
  denote the factorization of $\mathfrak{a}$ into the product of 
  powers of distinct prime ideals.
  By the Chinese remainder theorem, there exists an element $\tilde r_{\alpha}$ in 
  $\mathcal{O}_K$ which goes to $s_{\mathfrak{p}_i^{v_i}}$ under 
  the projection map from $\mathcal{O}_K$ to 
  $\mathcal{O}_K / \mathfrak{p}_i^{v_i}$
  for every $1 \leq i \leq s$.
  It follows that 
  $q(\tilde r_{\alpha})$ lies in $\mathfrak{a}$. 
  Since $s_{\mathfrak{p}_i^{m}} \equiv s_{\mathfrak{p}_i^{m+1}} 
  \pmod{\mathfrak{p}_i^{m}}$ holds for every $m \geq 1$, we get that
  \[
    \tilde r_{\alpha} \equiv s_{\mathfrak{p}_{i}^{n}} \pmod{\mathfrak{p}_{i}^{n}}
  \]
  holds for every $1 \leq n \leq v_i$ and $1 \leq i \leq s$.
  
  Note that $\det(\mathcal{A})$ is a nonzero integer.
  Using \(\mathcal{A} \cdot \mathrm{adj}(\mathcal{A}) = \det(\mathcal{A})I_d\), 
  it follows that given any vector $v$ in $\mathbb{Z}^{d}$,
  there exists a vector $u = (u_1, \dots, u_{d})^t$ in
  $\mathbb{Z}^{d}$ satisfying $0 \leq u_i < |\det(\mathcal{A})|$ for every
  $1 \leq i \leq d$, such that 
  $u + \mathcal{A} \mathbb{Z}^{d} = v + \mathcal{A} \mathbb{Z}^{d}$ holds.
  This implies that there exists an element 
  \(r_\alpha\) 
  in \(\mathcal{O}_K\) 
  such that 
  \(r_\alpha \equiv \tilde r_\alpha \pmod{\alpha}\) 
  and \(r_\alpha\) lies in \(S(\det(\mathcal{A}))\).
  This completes the proof. 
\end{proof}

We follow the strategy of Lucier discussed in \cite{Lucier}, to construct 
an auxiliary polynomial. 
We state the analogous properties for this auxiliary polynomial 
and give a proof for the sake of completeness.

Let $r_{\alpha}$ be an element
of $S((d^2 E M_{\alpha})^{d})$
satisfying the properties as in \cref{existence of r_a}.
It follows that $q(r_{\alpha} + \alpha x)$ is a polynomial 
with coefficients in $\mathfrak{a}$.
Define the polynomial $q_{\alpha}(x) \in \mathcal{O}_K[x]$ as
\[
  q_{\alpha}(x)
  =
  \frac{q(r_{\alpha} + \alpha x)}{\alpha}.
\]

\begin{lemma} \label[lemma]{intersectivity of auxiliary polynomial}
  The polynomial $q_{\alpha}(x)$ is an intersective polynomial 
  of degree $\kappa$, and satisfies
  \[
    M_{q_\alpha} \leq 
    C_3 M_\alpha^{4d \kappa} M_q.
  \]
\end{lemma}

\begin{proof}
  Let $\mathfrak{p}$ be a nonzero prime ideal of $\mathcal{O}_K$.
  Denote the localization of 
  $\mathcal{O}_K$ at $\mathfrak{p}$
  by $\mathcal{O}_{K,\mathfrak{p}}$, 
  and let $\varpi$ be a uniformizer for $\mathfrak{p} \mathcal{O}_{K,\mathfrak{p}}$.
  First, let us consider the case when $\alpha$ does not lie in $\mathfrak{p}$.
  Note that $\mathfrak{p}^n + \mathfrak{a} = \mathcal{O}_K$ holds for every
  positive integer $n$.
  It follows that $\alpha$ is a unit of $\widehat{\mathcal{O}}_{K,\mathfrak{p}}$.
  Hence, for some element $\alpha'$ in $\widehat{\mathcal{O}}_{K,\mathfrak{p}}$, 
  we have $\alpha \alpha' = 1$.
  Note that $\alpha' (z_{\mathfrak{p}} - r_{\alpha})$ is a root 
  of $q_{\alpha}(x)$ in $\widehat{\mathcal{O}}_{K,\mathfrak{p}}$.

  Now, consider the case when $\alpha$ lies in $\mathfrak{p}$.
  Let \(m\) denote the nonnegative integer such that 
  \[
    \alpha \mathcal{O}_{K,\mathfrak{p}} 
    = \mathfrak{p}^m \mathcal{O}_{K,\mathfrak{p}}
  \]
  holds. Hence, there exists a unit $\beta$ in $\mathcal{O}_{K,\mathfrak{p}}$ 
  such that 
  \[
    \alpha = \varpi^m \beta.
  \]
  Let $\{s_k\}_{k \geq 1}$ be a sequence in $\mathcal{O}_K$ such that
  $s_k$ goes to $s_{\mathfrak{p}^k}$ under the projection map from
  $\mathcal{O}_K$ to $\mathcal{O}_K / \mathfrak{p}^k$ for all $k \geq 1$.
  We claim that there exists a sequence $\left\{ x_n \right\}_{n \geq 1}$ 
  in $\mathcal{O}_{K,\mathfrak{p}}$ satisfying 
  \[
    r_{\alpha} + \alpha x_n 
    = s_{m + n},
    \quad 
    x_{n} \equiv x_{n+1} \pmod{\mathfrak{p}^{n} \mathcal{O}_{K,\mathfrak{p}}}
  \]
  for all $n \geq 1$.
  To prove it, let us 
  fix a positive integer $n$.
  Note that $s_{m+n} \equiv r_{\alpha} \pmod{\mathfrak{p}^m}$, and hence, 
  $s_{m+n} \equiv r_{\alpha} \pmod{\mathfrak{p}^m \mathcal{O}_{K,\mathfrak{p}}}$,
  which shows that there exists an element $v_n$ in $\mathcal{O}_{K,\mathfrak{p}}$ 
  such that $s_{m+n} = r_{\alpha} + v_n \varpi^m$.
  Taking \(x_n = \beta^{-1} v_n\), 
  we obtain 
  \begin{equation} \label{solution to the equation}
    r_{\alpha} + \alpha x_n 
    = r_{\alpha} + v_n \varpi^m
    = s_{m + n}.
  \end{equation}
  This yields a sequence $\left\{ x_n \right\}_{n \geq 1}$ in 
  $\mathcal{O}_{K,\mathfrak{p}}$ satisfying 
  \[
    r_{\alpha} + \alpha x_n 
    = s_{m + n}
  \]
  for all $n \geq 1$.
  The congruence
  \begin{align*}
    s_{m+n} & \equiv s_{m+n+1} 
    \pmod{\mathfrak{p}^{m+n} \mathcal{O}_{K,\mathfrak{p}}},
  \end{align*}
  implies $\alpha x_{n} \equiv \alpha x_{n+1} 
  \pmod{\mathfrak{p}^{m+n} \mathcal{O}_{K,\mathfrak{p}}}$, 
  which gives that 
  \[
    \varpi^m \beta x_{n} \equiv \varpi^m \beta x_{n+1} 
    \pmod{\mathfrak{p}^{m+n} \mathcal{O}_{K,\mathfrak{p}}}.
  \]
  It follows that $x_{n} \equiv x_{n+1} 
  \pmod{\mathfrak{p}^{n} \mathcal{O}_{K,\mathfrak{p}}}$.
  This completes the proof of the Claim.
 
  For any \(n \geq 1\), 
  let $\overline{x}_n$ denote the image of $x_n$ under the 
  composition of the 
  projection
  map $\mathcal{O}_{K,\mathfrak{p}}
  \to 
  \mathcal{O}_{K,\mathfrak{p}}/\mathfrak{p}^n \mathcal{O}_{K,\mathfrak{p}}$
  and the isomorphism
  \(
    \mathcal{O}_{K,\mathfrak{p}}/\mathfrak{p}^n \mathcal{O}_{K,\mathfrak{p}}
    \simeq 
    \mathcal{O}_K/\mathfrak{p}^n
  \).
  Put $z = (\overline{x}_n)_{n \geq 1}$.
  Since 
  $x_{n} \equiv x_{n+1} \pmod{\mathfrak{p}^{n} \mathcal{O}_{K,\mathfrak{p}}}$ 
  holds, 
  it follows that $\overline{x}_{n+1}$ goes to $\overline{x}_n$ 
  under the projection map from $\mathcal{O}_K / \mathfrak{p}^{n+1}$
  to $\mathcal{O}_K / \mathfrak{p}^n$, for all $n \geq 1$, and hence,
  $z$ is an element of 
  $\widehat{\mathcal{O}}_{K,\mathfrak{p}}$.
  Note that 
  \(
    \alpha  
    q_\alpha(x_n)
    = 
    q(s_{m + n})
  \)
  holds in \(\mathcal{O}_{K, \mathfrak{p}}\).
  Since \(q(s_{\mathfrak{p}^{m + n}})\) is equal to \(0\) 
  in \(\mathcal{O}_K/\mathfrak{p}^{m + n}\),
  it follows that \(q_\alpha(\overline{x}_n)\) is equal to 
  \(0\) in \(\mathcal{O}_K/\mathfrak{p}^n\).
  This shows that $z$ is a root of $q_{\alpha}(x)$ in 
  $\widehat{\mathcal{O}}_{K,\mathfrak{p}}$.
  Consequently, the polynomial $q_{\alpha}(x)$ is an intersective polynomial.

  Write $q(x) = c_0 + c_1 x + \dots + c_\kappa x^\kappa$,
  and hence,
  \[
    q_{\alpha}(x) =
    \frac{1}{\alpha} 
    \sum_{j=0}^{\kappa}
    \left(
      \sum_{i=j}^{\kappa}
      c_i \binom{i}{j} r_{\alpha}^{i-j} \alpha^{j}
    \right)
    x^j.
  \]
  Using \cref{expression of alpha}, we obtain that
  \[
    q_{\alpha}(x) =
    \frac{(-1)^{d-1} f(\alpha)}{\det(A)}
    \sum_{j=0}^{\kappa}
    \left(
      \sum_{i=j}^{\kappa}
      c_i \binom{i}{j} r_{\alpha}^{i-j} \alpha^{j}
    \right)
    x^j.
  \]

  Note that the coefficients of the polynomial $f(x)$ are integers 
  and have absolute value
  at most $(d^2 E M_{\alpha})^d$. 
  Using \(\alpha \in S(M_\alpha)\) and \cref{Eqn: product of finite lattices}, 
  we obtain 
  \begin{align*}
    \alpha^j 
    &\in S(d^{2(j-1)} E^{j-1} M_\alpha^j) 
    \subseteq S(d^{2(\kappa-1)} E^{\kappa-1} M_\alpha^\kappa)
    \text{ for all } 1 \leq j \leq \kappa,
  \end{align*}
  which implies that 
  \begin{align*}
    f(\alpha) 
    &\in 
    S \left(
      \kappa d^{2(\kappa + d -1)} 
      E^{\kappa + d -1} M_\alpha^{\kappa + d}
    \right).
  \end{align*}
  Similarly, for \(0 \leq j \leq i \leq \kappa\),  
  using 
  \[
    r_\alpha \in S((d^2 E M_{\alpha})^{d}),
  \]
  and applying \cref{Eqn: product of finite lattices}, 
  we obtain 
  \begin{align*}
    r_\alpha^{i - j}
    & \in S(d^{2(\kappa + \kappa d -1)} E^{\kappa + \kappa d-1} 
    M_{\alpha}^{\kappa d})
    ,
  \end{align*}
  and combining it with the above yields 
  that 
  \[
    c_i r_\alpha^{i - j} \alpha^j
    \in 
    S
    \left( 
      d^{2(2\kappa + \kappa d)} E^{2\kappa + \kappa d} 
      M_{\alpha}^{\kappa d + \kappa}  
      M_q
    \right)
  \]
  For any \(0 \leq j \leq \kappa\), 
  it follows that 
  \[
    \frac{(-1)^{d-1} f(\alpha)}{\det(A)}
    \left(
      \sum_{i=j}^{\kappa}
      c_i \binom{i}{j} r_{\alpha}^{i-j} \alpha^{j}
    \right)
    \in 
    S
    \left( 
      \kappa 2^\kappa
      d^{2(3\kappa + \kappa d + d)} E^{3\kappa + \kappa d + d} 
      M_{\alpha}^{\kappa d + 2\kappa + d}  
      M_q
    \right).
  \]
  We conclude that the coefficients of \(q_\alpha\) 
  lie in \(S(C_3 M_\alpha^{4d\kappa} M_q)\).
\end{proof}

A consequence of \cite[Lemma~23]{Lucier} is mentioned in 
\cite[Proposition~6.3]{RiceThesis}.
An analog of that for the ring $\mathcal{O}_K$ is as follows.

\begin{lemma} \label[lemma]{Lucier}
  Let $A$ be a subset of $\mathcal{O}_K$ and assume that 
  $
    (A - A)\cap q(\mathcal{O}_K) \subseteq \{0\}.
  $
  Let $n \in \mathcal{O}_K$ and $A'$ be a subset of 
  $\{x \in \mathcal{O}_K \colon n + \alpha x \in A\}$.
  Then
  $
    (A' - A') \cap q_\alpha({\mathcal{O}_K}) \subseteq \{0\}
  $
  holds. 
\end{lemma}
\begin{proof}
  If there were two distinct elements 
  $a_1, a_2$ of $A'$ 
  satisfying
  \[
    a_1 - a_2 = q_\alpha(v)
  \]
  for some element $v \in \mathcal{O}_K$,
  then it would follow that 
  \begin{align*}
    (n + \alpha a_1) - (n + \alpha a_2)
    & = \alpha (a_1 - a_2)\\
    & = \alpha q_\alpha(v)\\
    & = q(r_{\alpha} + \alpha v),
  \end{align*}
  which contradicts the assumption.
  Hence, the result is proved.
\end{proof}

\section{Proof of the main result}
\label{section:proof of main result}

In this section, we give a proof of \cref{Thm:main result}.
For a given nonconstant polynomial $p(x)$ of degree $\kappa$, 
an integer $r \geq 2$, and $\delta \in (0,1]$, 
consider the integer $N_r(\delta, M_p)$ given by
\[
  N_r(\delta, M_p)
  =
  \left\lceil
    \left(
      \frac{2^{3 \cdot 2^{\kappa-1} + d + 5} C_1 C_2}
      {\delta^{2^{\kappa - 1} + 1}}
    \right)^{8\kappa^{r}}
  \right\rceil
  M_{p}^4.
\]
Note that for a given polynomial $p(x)$, if $\delta \geq \delta'$
then $N_r(\delta, M_p) \leq N_r(\delta', M_p)$.

Considering this, and combining \cref{density increment lemma} with \cref{Lucier},
we can state the following result for 
intersective polynomials over $\mathcal{O}_K$.

\begin{corollary} \label[corollary]{Combination of DIL and Lucier}
  Let $q(x)$ be a nonconstant intersective polynomial of degree $\kappa$ in 
  $\mathcal{O}_K[x]$, and $\delta_0 \in (0,1]$. Let $N$ be a positive integer 
  and $A$ be a nonempty subset of $S(N)$ of density $\delta \geq \delta_0$, 
  such that 
  \[
    (A - A) \cap q(\mathcal{O}_K) \subseteq \{0\}.
  \]
  If for some integer $r \geq 2$, we have $N \geq N_r(\delta_0, M_q)$, 
  then there exist an intersective polynomial $q'(x)$ of degree $\kappa$, 
  satisfying
  \[
    M_{q'} \leq 
    C_4
    M_q^{5\kappa d}
    N^{d/\kappa^{r-3}},
  \]
  and a subset $A'$ of $S(N^{1/2\kappa})$ of density at least $\delta + c\delta^C$, 
  such that
  \[
    (A' - A') \cap q'(\mathcal{O}_K) \subseteq \{0\}
  \]
  holds, where $c, C$ denote the constants as in \cref{density increment lemma}.
\end{corollary}

\begin{proof}
  Let $r \geq 2$ be an integer and assume that $N \geq N_r(\delta_0, M_q)$ holds.
  Applying \cref{density increment lemma} gives that there exist
  elements $n$ and $\alpha$ in $\mathcal{O}_K$ with 
  \begin{equation} \label{bound of alpha}
    M_\alpha \leq 
    C_2^\kappa M_{q} N^{1/4\kappa^{r-2}},
  \end{equation}
  such that the set $L \coloneq n + \alpha S(N^{1/2\kappa})$ is contained in $S(N)$ and 
  the density of $A$ in $L$ is at least $\delta + c \delta^C$,
  for some constants $c,C > 0$ depending only on $\kappa$.
  Let us take
  $A' = \left\{ x \in S(N^{1/2\kappa}) \colon n + \alpha x \in A \right\}$,
  and note that the density of $A'$ in 
  $S(N^{1/2\kappa})$ is at least $\delta + c \delta^C$. 
  Moreover, \cref{Lucier} implies that 
  \[
    (A' - A') \cap q_\alpha(\mathcal{O}_K) \subseteq \{0\}.
  \]

  Write $q'(x) \coloneq q_\alpha(x)$.
  By \cref{intersectivity of auxiliary polynomial}, the polynomial 
  $q'(x)$ is an intersective polynomial of degree $\kappa$,
  and it satisfies
  \[
    M_{q'} \leq 
    C_3 M_\alpha^{4\kappa d} M_q.
  \]
  Using \cref{bound of alpha} in above, we obtain that
  \[
    M_{q'} \leq 
    C_2^{4 \kappa^2 d} C_3
    M_q^{5\kappa d}
    N^{d/\kappa^{r-3}}.
  \]
  Hence, \cref{Combination of DIL and Lucier} is proved.
\end{proof}

Multiple application of the above corollary proves our main result.
We start by assuming that $N$ is larger than a quantity depending on the 
density of the considered set, and reach to a contradiction by obtaining a subset of 
density larger than $1$. This yields that $N$ is not very large,
and consequently, we get the desired bound for the density of the set.

\begin{proof}[Proof of \cref{Thm:main result}]
  Let $c, C$ denote the constants as in 
  \cref{density increment lemma}.
  Let $\delta$ denote the density of $A$.
  Write $\delta_0 = \delta$, and 
  consider the sequence $\{\delta_i\}_{i \geq 0}$ 
  satisfying $\delta_{i+1} = \delta_{i} + c\delta_{i}^C$
  for all \(i \geq 0\).
  Let $t$ denote the smallest positive integer such that $\delta_t > 1$.
  Let us put $q_0(x) = q(x)$, 
  $r = \left\lceil \log_{\kappa} \left(8t (10 d)^{t}\right)\right\rceil + 2t + 2$ 
  and
  \begin{align*}
    N_0 =
    \left(
      \frac{2^{3 \cdot 2^{\kappa-1} + d + 6} C_1 C_2 C_4 M_q}
      {\delta^{2^{\kappa - 1} + 1}}
    \right)^{8t(5\kappa d)^{r + 2t}}.
  \end{align*}
  We claim that $N < N_0$ holds. 
  Let us assume on the contrary that $N \geq N_0$.
  Since $N_0 \geq N_r(\delta_0, M_{q_0})$, 
  consequently, $N \geq N_r(\delta_0, M_{q_0})$.
  Applying \cref{Combination of DIL and Lucier} 
  yields an intersective polynomial $q_1(x)$ of degree $\kappa$
  satisfying
  \[
    M_{q_1} \leq 
    C_4 M_q^{5\kappa d} N^{d/\kappa^{r-3}},
  \]
  and a subset
  $A_1$ of $S(N^{1/2\kappa})$ having density at least $\delta_1$
  such that 
  \[
    (A_1 - A_1) \cap q_1(\mathcal{O}_K) \subseteq \{0\}
  \]
  holds.
  Next, we show by induction that using the assumption $N \geq N_0$, 
  \cref{Combination of DIL and Lucier}
  can be applied $t$ times, to produce a subset of a suitable set, having 
  density larger than $1$, 
  and as a consequence, $N < N_0$ follows.

  Let $1 \leq i \leq t - 1$ be an integer, 
  and assume that \cref{Combination of DIL and Lucier}
  can be applied iteratively $i$ times.
  Therefore, for every \(1 \leq j \leq i\), we have 
  $N^{1/(2\kappa)^{j - 1}} \geq N_r(\delta_{j - 1}, M_{q_{j - 1}})$, and
  there exist an intersective polynomial $q_j(x)$ of degree $\kappa$
  satisfying 
  \begin{equation}
    \label{Eq: Mqj bound}
    M_{q_j} \leq C_4
    M_{q_{j-1}}^{5\kappa d} N^{d/\kappa^{r-3}(2\kappa)^{j - 1}},
  \end{equation}
  and a subset $A_j$ of the set $S(N^{1/(2\kappa)^j})$ 
  of density at least $\delta_j$ with 
  \[
    (A_j - A_j) \cap q_j(\mathcal{O}_K) \subseteq \{0\}.
  \]

  We claim that 
  \begin{equation}
    \label{Eq:Claim for applying induction}
    N^{1/(2\kappa)^i} \geq N_r(\delta_i, M_{q_i})
  \end{equation}
  holds. 
  Using \eqref{Eq: Mqj bound}, 
  we obtain that 
  \begin{align*}
    M_{q_i}
    & \leq 
    C_4^{1 + 5\kappa d + \dots + (5\kappa d)^{i-1}} M_{q}^{(5\kappa d)^{i}} 
    N^{\frac{d}{\kappa^{r-3}}
    \left( 
      \frac{1}{(2\kappa)^{i-1}} + \frac{5\kappa d}{(2\kappa)^{i-2}} + \dots +
      (5\kappa d)^{i-1}
    \right)}\\
    & \leq
    C_4^{t (5\kappa d)^{t}} M_{q}^{(5\kappa d)^{t}} 
    N^{\frac{t (5d)^t}{\kappa^{r-t-2}}}.
  \end{align*}
  Using $N \geq N_0$, the above bound and 
  $r = \left\lceil \log_{\kappa} \left(8t (10 d)^{t}\right)\right\rceil + 2t + 2$, 
  we obtain 
  \begin{align*}
    &
    \left(
      \left\lceil
        \left(
          \frac{2^{3 \cdot 2^{\kappa-1} + d + 5} C_1 C_2}
          {\delta^{2^{\kappa - 1} + 1}}
        \right)^{8\kappa^{r}}
      \right\rceil
      M_{q_i}^4
    \right)^{(2\kappa)^i}\\
    \leq & 
    \left(
      2
      \left(
        \frac{2^{3 \cdot 2^{\kappa-1} + d + 5} C_1 C_2}
        {\delta^{2^{\kappa - 1} + 1}}
      \right)^{8\kappa^{r}}
      C_4^{4t (5\kappa d)^{t}} M_{q}^{4(5\kappa d)^{t}} 
      N^{\frac{4t (5d)^t}{\kappa^{r-t-2}}}
    \right)^{(2\kappa)^t}\\
    \leq &
    \left(
      \frac{2^{3 \cdot 2^{\kappa-1} + d + 6} C_1 C_2 C_4 M_q}
      {\delta^{2^{\kappa - 1} + 1}}
    \right)^{4t(5\kappa d)^{r + 2t}}
    N^{\frac{4t (10 d)^{t}}{\kappa^{r-2t-2}}}\\
    \leq &
    N_0^{1/2}N^{1/2}\\
    \leq & N.
  \end{align*}
  This proves the claim that \eqref{Eq:Claim for applying induction} holds. 
  By induction, it follows that 
  \cref{Combination of DIL and Lucier}
  can be applied $t$ times, and this yields 
  a subset $A_t$ of $[N^{1/(2\kappa)^t}]$ 
  of density greater than $1$, 
  which is impossible. 
  This shows that 
  \begin{align*}
    N 
    &\leq
    \left(
      \frac{2^{3 \cdot 2^{\kappa-1} + d + 6} C_1 C_2 C_4 M_q}
      {\delta^{2^{\kappa - 1} + 1}}
    \right)^{8t(5\kappa d)^{r + 2t}}
  \end{align*}
  holds.
  Let us denote $c = 1/2^{3 \cdot 2^{\kappa-1} + 2}$, $C = 2^{\kappa-1} + 1$ and 
  note that if $n \geq 1/c\delta^C$, then $\delta_n > 1$.
  This implies that 
  \[
    t \leq \frac{1}{c\delta^C}.
  \] 
  Using 
  $r = \left\lceil \log_{\kappa} \left(8t (10 d)^{t}\right)\right\rceil 
  + 2t + 2$, 
  we obtain 
  \[
    r + 2t
    \leq \frac{10 d + 15}{c\delta^C}.
  \]
  This yields that 
  \begin{align*}
    \log \log N 
    & \leq 
    \frac{C_{q, d, E}}{\delta^C}
  \end{align*}
  holds, where $C_{q, d, E}$ is a constant depending on the polynomial $q(x)$,
  $d$ and $E$ only.
  Hence, 
  we obtain that 
  \[
    \delta
    \ll_{q, d, E}
    \frac{1}{\left(\log \log N\right)^{1/C}}
  \]
  holds. This proves \cref{Thm:main result}.
\end{proof}

\section*{Acknowledgements}

The first author is grateful to Saurabh Kumar Shrivastava for 
suggesting the problem and having fruitful discussions, 
and for his constant support during the preparation of this article.
He acknowledges the fellowship from the Council of Scientific and 
Industrial Research (CSIR), India, file no.~09/1020(17125)/2023-EMR-I,
during the Ph.D. program at IISER Bhopal.


\end{document}